\documentclass[]{amsart}

\usepackage{afterpage}
\usepackage{amsaddr}
\usepackage{amsfonts}
\usepackage{amsmath}
\usepackage{amsthm}
\usepackage{bbm}
\usepackage{euscript}
\usepackage[top=1.4 in, bottom=1in,left=1in,right=1in]{geometry}
\usepackage{graphicx}
\usepackage{hyperref}
\usepackage{mathrsfs}
\usepackage[numbers]{natbib}
\usepackage{subcaption}
\usepackage{tikz}
\usepackage{tikz-cd}
\usepackage{verbatim}

\graphicspath{ {./} }

\renewcommand{\c}{\mathscr{C}}
\newcommand{\C}{\mathbf{C}}
\newcommand{\D}{\mathbb{D}}

\newcommand{\E}{\mathbb{E}}

\newcommand{\F}{\mathscr{F}}
\newcommand{\G}{\mathscr{G}}
\newcommand{\hgt}{{\operatorname{\mathbf{ht}}}}

\renewcommand{\P}{\mathbb{P}}
\newcommand{\R}{\mathbb{R}}

\newcommand{\V}{\mathbf{V}}

\newcommand{\w}{\omega}

\newcommand{\X}{\mathbf{X}}

\newcommand{\Z}{\mathbf{Z}}

\newcommand{\csn}{\operatorname{\mathbf{csn}}}
\newcommand{\dist}{\operatorname{{dist}}}
\newcommand{\eps}{\varepsilon}
\newcommand{\fl}[1]{\lfloor #1 \rfloor}

\newcommand{\Var}{\operatorname{{Var}}}

\newcommand{\<}{\langle}
\renewcommand{\>}{\rangle}

\newcommand{\ER}{Erd\H{o}s-R\'{e}nyi }
\newtheorem{thm}{Theorem}[section]
\newtheorem{conj}[thm]{Conjecture}
\newtheorem{cor}[thm]{Corollary}
\newtheorem{lem}[thm]{Lemma}
\newtheorem{prop}[thm]{Proposition}

\theoremstyle{definition}

\theoremstyle{remark}
\newtheorem{remark}{Remark}[section]

\title[\ER graphs and Brownian motion]{A new relationship between Erd\H{o}s-R\'{e}nyi graphs, epidemic models and Brownian motion with parabolic drift}
\author{David Clancy, Jr. }
\address{Department of Mathematics, University of Washington}
\email{\href{mailto:djclancy@uw.edu}{djclancy@uw.edu}}
\begin{document}
\begin{abstract}
	In the Reed-Frost model, an example of an SIR epidemic model, one can examine a statistic that counts the number of concurrently infected individuals. This statistic can be reformulated as a statistic on the \ER random graph $G(n,p)$. Within the critical window of Aldous \cite{A97_1} and Martin-L\"{o}f \cite{MartinLof98}, i.e. when $p = p(n) = n^{-1}+\lambda n^{-4/3}$, this statistic converges weakly to a Brownian motion with parabolic drift stopped upon reaching a level. The same statistic exhibits a deterministic scaling limit when $p = (1+\lambda \eps_n)/n$ whenever $\eps_n\to 0$ and $n^{1/3}\eps_n\to\infty$. 
\end{abstract}	

\maketitle

\section{Introduction}

In this paper we provide a new relationship between an \ER random graph $G(n,p)$ when $n\to\infty$ with $p = p(n) = n^{-1}+\lambda n^{-4/3}$ and a Brownian motion with parabolic drift, $\X^\lambda = (\X^\lambda(t);t\ge 0)$, defined by
\begin{equation}\label{eqn:xlamb}
\X^\lambda(t) = B(t)+\lambda t - \frac{1}{2} t^2, 
\end{equation} for a standard Brownian motion $B$. The connection between this asymptotic regime and a Brownian motion with parabolic drift dates back to Aldous' work in \cite{A97_1} and the independent work of Martin-L\"{o}f \cite{MartinLof98}. The latter reference relies on the connection between \ER random graphs and the so-called Reed-Frost model for epidemics. The results presented below have implications for the Reed-Frost model as well.

The Reed-Frost model is an SIR model - that is individuals are either Susceptible to the disease, Infected with the disease, or have Recovered from the disease (sometimes called Removed). At time $t = 0$, there is some number of initially infected individuals $I_0$ in a population of size $n$. Consequently, there are $S_0 = n-I_0$ susceptible individuals. At time $t = 0,1,2,\dotsm$, each of the $I_t$ infected individuals infects each of the $S_t$ susceptible individuals with probability $p$. The susceptible individuals who become infected at time $t$ make up the $I_{t+1}$ infected individuals at time $t+1$. The connection between the Reed-Frost model and \ER random graph $G \sim G(n,p)$ is explained in \cite{BM90}. In brief, the initially infected individuals are uniformly selected vertices without replacement. The neighbors of infected individuals at time $t$, who have not already been infected, become the infected individuals at time $t+1$.

The structure of large random graphs has been an object of immense research dating back to the 1960s. One of the simplest models is the \ER random graph $G(n,p)$ on $n$ vertices where each of the $\displaystyle \binom{n}{2}$ possible edges is independently added with probability $p$. In their original work \cite{ER60}, Erd\H{o}s and R\'{e}nyi show that if $p = p(n) = c/n$ for some constant $c$ then the following phase shift occurs
\begin{enumerate}
\item if $c<1$ the largest component is of order $\Theta(\log n)$;
\item if $c>1$ the largest component is of order $\Theta(n)$ and the second largest component is of order $\Theta(\log n)$;
\item if $c = 1$ then the two largest components are of order $\Theta(n^{2/3})$. 
\end{enumerate} This has a corresponding interpretation for Reed-Frost model: the largest components of the \ER graph represent the size of largest possible outbreaks in the Reed-Frost model when only a single individual is initially infected. 

Much interest has been paid to the phase shift that occurs at and around $c =1$. In the critical window $p(n) = n^{-1}+\lambda n^{-4/3}$ for a real parameter $\lambda$, the size of the components of the random graph were established in \cite{A97_1} and are related to the excursion lengths of a Brownian motion with parabolic drift. More formally, let $\X^\lambda =  \left(\X^\lambda(t);t\ge0\right)$ be a Brownian motion with parabolic drift defined by \eqref{eqn:xlamb}, and let $\gamma^\lambda(1)\ge \gamma^\lambda(2)\dotsm$ denote the lengths of the excursions of $\X^\lambda$ above its past infimum ordered by decreasing lengths. Then if $\c_n(1),\c_n(2),\dotsm$ are the components of $G(n,n^{-1}+\lambda n^{-4/3})$ ordered by decreasing cardinality there is convergence in distribution
\begin{equation}\label{eqn:compSize}
\left( n^{-2/3} \# \c_n(1), n^{-2/3} \# \c_n(2),\dotsm\right) \Longrightarrow \left(\gamma^\lambda(1),\gamma^\lambda (2),\dotsm \right)
\end{equation} with respect to the $\ell^2$-topology. 

The results in this paper are motivated by a question posed by David Aldous to the author during a presentation of the author's results in \cite{Clancy19}. The main results of \cite{Clancy19} relate the scaling limit of two statistics on a random forest model to the integral of an encoding L\'{e}vy process without negative jumps. The connection relies on a breadth-first exploration of the random forest. Aldous \cite{A97_1} used a breadth-first exploration to obtain the relationship between the \ER random graph $G(n,n^{-1}+\lambda n^{-4/3})$. Aldous asked if there was some relationship between analogous statistics on the graph $G(n,n^{-1}+\lambda n^{-4/3})$ and the integral of the Brownian motion with parabolic drift in equation \eqref{eqn:xlamb}. The answer to the question is yes and is provided with Theorem \ref{thm:kconv} and Theorem \ref{thm:kconv_gen} below.

\subsection{Statement of Results}\label{sec:statements}

Fix a real parameter $\lambda$, and define $\G_n = G(n,n^{-1}+\lambda n^{-4/3})$. Fix a $k\le n$ and uniformly choose $k$ vertices without replacement in the \ER graph $\G_n$, and denote these by $\rho_n(1),\rho_n(2),\dotsm, \rho_n(k)$. Let $\dist(-,-)$ denote the graph distance on $\G_n$ with the convention $\dist(w,v) = \infty$ if $w$ and $v$ are in distinct connected components. 

For each vertex $v\in \G_n$, define the height of a vertex, denoted by $\hgt^k_n(v)$, by
\begin{equation*}
\hgt^k_n(v) = \min_{j\le k} \dist(\rho_n(j),v).
\end{equation*} 

We remark that applying a uniformly chosen permutation to the vertex labels in an \ER graph $G(n,p)$ gives an identically distributed random graph and so we could take the vertices $\{\rho_n(1),\dotsm, \rho_n(k)\}$ to simply be the vertices $\{1,\dotsm,k\}$. With this observation, using $\{\rho_n(1),\dotsm, \rho_n(k)\}$ instead of $\{1,2,\dotsm,k\}$ may seem like an unnatural choice in terms of the \ER random graph. If we instead think of the corresponding SIR epidemic model -- more specifically the Reed-Frost model --  this choice becomes much more natural. Indeed, these vertices $\rho_n(1),\dotsm, \rho_n(k)$ become the $k$ initially infected individuals in a population of size $n$.

We define the process $Z_n^k = \left( Z_n^k(h);h = 0,1,\dotsm\right)$ by 
\begin{equation}\label{eqn:zDiscrete}
Z_n^k(h) = \# \{v\in \G_n: \hgt^k_n(v) = h\}.
\end{equation} In words, $Z_n^k(h)$ is the number of vertices at distance exactly $h$ from the $k$ uniformly chosen vertices $\rho_n(1),\dotsm, \rho_n(k)$. In terms of the corresponding SIR model, $Z_n^k(h)$ represents the number of individuals infected at ``time'' $h$ when $k$ individuals are infected at time $0$. 

The statistic we examine measures how many vertices in $\G_n$ are at the same distance from the $k$ uniformly chosen vertices. Namely, given a $k\le n$ and a vertex $v\in \G_n$ we define the statistic
\begin{equation*}
\csn^k_n(v) = \#\{w\in \G_n: \hgt_n^k(v) = \hgt_n^k(w)\},
\end{equation*} and call this the \textit{cousin statistic}. In a random forest model where a genealogical interpretation is more natural, the statistic was used in \cite{Clancy19} to count the number of ``cousin vertices.'' In the graph context this statistic seems like an unnatural choice. If we instead think of the epidemic model as described in the second paragraph of the introduction, then $\csn_n^k(v)$ becomes much more natural. The value of $\csn_n^k(v)$ represents the number of people infected at the same instance that individual $v$ is infected when $k$ individuals are infected at time $0$ and the total population is exactly $n$. 

Before discussing a scaling limit involving the cousin statistic, we introduce a labeling of the vertices
$$
\{v\in \G_n: \hgt_n^k(v)<\infty\},
$$ i.e. the vertices connected to one of the randomly chosen vertices $\rho_n(1),\dotsm, \rho_n(k)$. We label these vertices $w_n^k(0),w_n^k(1),\dotsm$ in any way that $j\mapsto \hgt_n^k(w_n^k(j))$ is non-decreasing. One such way is by first setting $w_n^k(0) = \rho_n(1)$, $w_n^k(1) = \rho_n(2)$, $\dotsm, w_n^k(k-1) = \rho_n(k)$. and then assigning labels inductively so that the unlabeled neighbors of $w_n^k(i)$ are assigned labels before the unlabeled neighbors of $w_n^k(j)$ for $i<j$. Since $\csn_n^k(w_n^k(j))$ only depends on the height $\hgt(w_n^k(j))$, the specific ordering of neighbors within the same height is not of much importance. In terms of the epidemic model that we've mentioned several times already, the ordering $w_n^k(0), w_n^k(1),\dotsm$ orders the total number of infected individuals in terms of who got infected first. 

We define the cumulative cousin process
\begin{equation} \label{eqn:kDef}
K_n^k(j) = \sum_{i=0}^{j-1} \csn_n^k(w_n^k(i)).
\end{equation} Eventually there will be no vertex labeled $w_n^k(i)$ in the graph, i.e. we have exhausted all vertices in a connected component containing $\rho_n(k)$. At which point we just define $\csn_n^k(w_n^k(i)) = 0$. 
The following theorems describe the scaling limit of the cousin statistic $\csn$ and the cumulative sum $K_n^k$. 
In the regime studied by Aldous \cite{A97_1}:
\begin{thm}\label{thm:kconv} Fix a $\lambda\in \R$ and consider the graph $\G_n = G(n,n^{-1}+\lambda n^{-4/3})$.
	Fix an $x>0$ and let $k = k(n,x) = \fl{n^{1/3}x}$. Then the following convergence holds in the Skorohod space $\D(\R_+,\R_+)$:
	\begin{equation}\label{eqn:csnconv}
	\left(n^{-1/3}\csn_n^k(w_n^k(\fl{n^{2/3}t}));t\ge 0 \right) \Longrightarrow \left(x + \X^\lambda(t\wedge T_{-x}) ;t \ge 0 \right),
	\end{equation} where $\X^\lambda$ is a Brownian motion with parabolic drift in \eqref{eqn:xlamb} and $T_{-x} = \inf\{t: \X^\lambda (t) = -x\}$. 
\end{thm}

In a more general view of the critical window, which has been studied in, for example, \cite{DKLP10,DKLP11,Luczak98,RW10}, we have the following theorem
\begin{thm}\label{thm:kconv_gen}

	Consider the Erd\H{o}s-R\'{e}nyi random graph $\G_n^\eps:= G(n,(1+\lambda\eps_n)/n)$, where $\eps_n>0$, and $\eps_n\to 0$ but $\eps_n^3 n\to\infty$. Let $k = k(n,x) = \fl{\eps_n^2 n x}$. Then, for this sequence of graphs, the following convergence holds on the Skorohod space
	\begin{equation*}
\left(n^{-1}\eps_n^{-2} \csn_n^k(w_n^k(\fl{n\eps_n t}));t\ge0 \right) \Longrightarrow \left((x+\lambda t - \frac{1}{2}t^2)\vee 0; t\ge0 \right).
	\end{equation*}

\end{thm}

The proof of the Theorem \ref{thm:kconv} above can be found in Section \ref{sec:klim} and the proof of Theorem \ref{thm:kconv_gen} can be found in Section \ref{sec:theta}. The proof relies heavily scaling limit for the process $Z_n^k$ and a time-change argument similar to the Lamperti transform. For the critical window in Theorem \ref{thm:kconv}, the scaling limit is known in the literature for continuous time epidemic models \cite{DL06} and \cite{Simatos15}. See also, \cite[Appendix 2]{vBML80}. We state it as the following lemma.
\begin{lem}[\cite{DL06,Simatos15}] \label{thm:zconv}
Fix an $x>0$ and let $k = k(n) = \fl{n^{1/3}x}$. Then, as $n\to\infty$, the following weak convergence holds in the Skorohod space $\D(\R_+,\R_+)$
\begin{equation*}
\left(n^{-1/3} Z_n^{k(n,x)}(\fl{n^{1/3}t}) ;t \ge 0\right) \Longrightarrow \left(\Z(t);t\ge 0 \right),
\end{equation*} where $\Z$ is the unique strong solution of the following stochastic equation
\begin{equation}\label{eqn:zsde1}
\Z(t) = x + \int_0^t \sqrt{\Z(s)}\,dW(s) + \left( \lambda - \frac{1}{2} \int_0^t \Z(s)\,ds\right)\int_0^t\Z(s)\,ds,
\end{equation} which is absorbed upon hitting zero and $W$ is a standard Brownian motion.
\end{lem}

In Section \ref{sec:z} we provide proofs of the lemmas needed to go from the continuous time statements in \cite{DL06,Simatos15} to the formulation in Lemma \ref{thm:zconv}. These results we prove will be used in Section \ref{sec:theta} to argue a similar scaling result as in Lemma \ref{thm:zconv} in a more general critical window. We do provide a proof the strong existence and uniqueness of solutions to the stochastic equation \eqref{eqn:zsde1} with Lemma \ref{lem:uniquenessLemma}. 

We also argue the following proposition for the number of vertices in the connected subset of the graph connected to one of the $k$ randomly chosen vertices (cf \cite[Theorem 1]{MartinLof98}).
\begin{prop}\label{prop:components}
	Let $k = k(n) = \fl{n \eps_n^2 x}$ where $\eps_n$ satisfies \eqref{eqn:thetan}. Let $A_n^\eps(k)$ denote the number of vertices in $\G_n^\eps = G(n,(1+\lambda \eps_n)/n)$ which are in the same connected component as some vertex in $\{1,2,\dotsm k\}$. Then if $\eps_n\to 0$ but $n\eps_n^3\to\infty$, for each $\eta>0$,
	\begin{equation*}
	\P\left(n^{-1/3} \eps_n A_n^\eps(k) >  \lambda+ \sqrt{\lambda^2+2x}-\eta \right) \longrightarrow 1,\qquad \text{as }n\to\infty.
	\end{equation*}
\end{prop}

We remark that the limiting process $\Z$ found in \eqref{eqn:zsde1} is precisely what one should expect from Aldous' convergence of a rescaled breadth-first walk towards \eqref{eqn:xlamb} found in \cite{A97_1} and the results connecting breadth-first walks on forests and height profiles in \cite{CPU13}. Using the time change $u(t)$ satsifies $\int_0^u\Z(s)\,ds = t$, one can see that the process $Y(t) = \Z(u)$ becomes a Brownian motion with parabolic drift killed upon hitting zero. This is further explained in Lemma \ref{lem:uniquenessLemma}. The connection is a random time-change called the Lamperti transform in the literature on branching processes. This Lamperti transformation has a natural interpretation which relates a breadth-first walk and a corresponding time-change which counts the number of cousin vertices. Moreover, this transformation gives a bijective relationship between a certain class of L\'{e}vy processes and continuous state branching processes. The bijection originated in the work of Lamperti \cite{Lamperti67}, but was proved by Silverstein \cite{Silverstein67}. For a more recent approach see \cite{CLU09}. See \cite{CPU13,CPU17} for generalizations and results involving scaling limits.

To view the connection with the author's previous work in \cite{Clancy19} we include the following corollary of Theorems \ref{thm:kconv} and \ref{thm:kconv_gen}. 
\begin{cor}\label{cor:2} Fix a $\lambda\in \R$. 
	\begin{enumerate}
		\item  Let $k = k(n,x) = \fl{n^{1/3}x}$.	Let $\G_n = G(n, n^{-1}+\lambda n^{-4/3})$, and let $K_n^k$ be the cumulative cousin process defined by \eqref{eqn:kDef}. Then on the Skorohod space
		\begin{equation}\label{eqn:kconv}
		\left(n^{-1} K_n^{k(n,x)}(\fl{n^{2/3}t}); t\ge0\right) \Longrightarrow \left(\int_0^{t\wedge T_{-x}} \left(x+ \X^\lambda(s) \right)\,ds; t \ge 0 \right).
		\end{equation} 
		\item Let $\eps_n$ be a sequence of strictly positive numbers such that $\eps_n\to 0,$ but $\eps_n^3 n\to \infty$. Let $k = k(n,x) = \fl{\eps_n^2 nx}$. Let $\G_n^\eps = G(n,(1+\lambda\eps_n)/n)$ and let $K_n^{\eps,k}$ be the cumulative cousin process defined by \eqref{eqn:kDef} for this sequence of graphs. Then on the Skorohod space $\D(\R_+,\R_+)$ we have
		\begin{equation*}
		\left(\frac{1}{\eps_n^3 n^{2}} K^{\eps,k(n,x)}_n (\fl{n\eps_n t}) ;t\ge 0 \right) \longrightarrow \left( \left(xt + \frac{1}{2}\lambda t^2 - \frac{1}{6}t^3 \right)\vee 0 ;t\ge 0 \right)
		\end{equation*} in probability.
	\end{enumerate}

\end{cor}

Let us take some time to discuss the connection between Theorem \ref{thm:kconv} and the results in \cite{Clancy19}. The work in \cite{Clancy19} originated in trying to give a random tree interpretation of various results connected to edge limits for eigenvalues of random matrices \cite{GS18, LS19}. In short, those works give a description of the 
random variable
\begin{equation*}
A = \sqrt{12} \left(\int_0^1 r(t)\,dt - \frac{1}{2} \int_0^\infty \left(L_1^v(r)\right)^2\,dv  \right),
\end{equation*} where $r = (r(t);t\in[0,1])$ is a reflected Brownian bridge and $L(r) = \left(L_t^v(r);t\in[0,1],v\ge 0\right)$ is its local time (see Chapter VI of \cite{RY99}). The appearance of the $\sqrt{12}$ term is simply a convenient scaling. In \cite{LS19}, the authors compute some moments of $A$ which led them to ``believe that $A$ admits an interesting combinatorial interpretation." Such an interpretation was given in \cite{Clancy19} involving comparisons of two statistics on random trees and forests, one of which is the number of ``cousin'' vertices $\csn(v)$.

\subsection{Overview of the Paper}

In Section \ref{sec:prelims}, we discuss some preliminaries on random graphs, the Reed-Frost model. In Section \ref{sec:labeling}, we discuss in more detail the ordering of vertices discussed briefly prior to Theorem \ref{thm:kconv}.

In Section \ref{sec:z}, we discuss some lemmas on the asymptotics of the Reed-Frost model. This allows us to go from the scaling limits of the continuous time SIR models in \cite{DL06, Simatos15} to the statement of Lemma \ref{thm:genz}. This section is focused on the nearly critical regime $p(n) = n^{-1}+\lambda n^{-4/3}$. In Section \ref{sec:zthm}, we prove the strong existence and uniqueness of solutions to equation \eqref{eqn:zsde1}. We also discuss the time-change discussed in the introduction.

In Section \ref{sec:klim}, we prove Theorem \ref{thm:kconv} using Lemma \ref{thm:zconv}. In Section \ref{sec:ss}, we prove a self-similarity result for the solution of stochastic differential equations \eqref{eqn:zsde1}. This is analogous to how Aldous \cite{A97_1} described the (time-inhomogeneous) excursion measure of the process $\X^\lambda$ in terms of the It\^{o} excursion measure.

In Section \ref{sec:theta}, we study the more general nearly critical window $(1+\lambda\eps_n)/n$ where $\eps_n\to 0$ and $\eps_n^3 n\to\infty$. In this section we generalize many of the lemmas in Section \ref{sec:z}, in order to prove Theorem \ref{thm:kconv_gen}.

\section{Preliminaries} \label{sec:prelims}

\subsection{Random Graphs}\label{sec:ER}

Recall that the \ER graph $G(n,p)$ is the graph on $n$ elements where each edge is independently included with probability $p$. The fundamental paper of Erd\H{o}s and R\'{e}nyi \cite{ER60} describes the size of the largest component as $n\to \infty$ and $p = \frac{c}{n}$. As briefly discussed in the introduction, a phase transition occurs at $c = 1$. In the subcritical case ($c<1$)  the largest component is of (random) order $\Theta(\log n)$ and in the supercritical case $(c>1)$ the largest component is of order $\Theta(n)$ while in the critical case ($c=1$) the largest two components are of order $\Theta(n^{2/3})$.

Much interest has be paid towards the phase transition which occurs near $c = 1$. Bollob\'{a}s \cite{B84a} showed that if $c = 1+ n^{-1/3}(\log n)^{1/2}$ then the largest component is of order $n^{2/3}(\log n)^{1/2}$. Later {\L}uckzak, Pittel and Wierman \cite{LPW94} showed that in the regime $c = 1+\lambda n^{-1/3}$ for some constant $\lambda$ then any component of the \ER graph has at most $\xi_n$ surplus edges, i.e. each component of size $m$ has at most $m-1+\xi_n$ edges and $\xi_n$ is bounded in probability as $n\to\infty$. Prior to the work in \cite{LPW94}, Bollob\'{a}s studied this regime in \cite{B84b}. See also the monograph \cite{B85}.

We now briefly recall some central results for the asymptotics of $\G_n$. Without going into all of the details, Aldous \cite{A97_1} encapsulates information on the size of the components in terms of a random walk $X_n = \left(X_n(j) ;j = 0,1,\dotsm \right)$. Namely, setting $T_n(\ell) = \inf\{j: X_n(j) = -\ell \}$, the sizes of the components of $\G_n$ are recovered by $(T_n(\ell+1) -T_n(\ell))$. Using this relationship along, with the scaling limit
\begin{equation*}
\left(n^{-1/3}X_n(\fl{n^{2/3}t});t\ge0 \right) \Longrightarrow \left( \X^\lambda(t);t\ge 0\right), 
\end{equation*} Aldous \cite{A97_1} is able to relate asymptotics of both the number of surplus edges and the size of the components to the excursions above past minima of the process $\X^\lambda$ defined in \eqref{eqn:xlamb}. 

Just as there is a theory of continuum limits of random trees (see, e.g. Aldous's work \cite{A90,A91,A93} and the monograph \cite{LD02}) there is a continuum limit of the largest components of $\G_n$. The scaling limit for the largest components was originally described by Addario-Berry, Broutin and Goldschmidt in \cite{ABG12} and additional results about continuum limit object can be found in their companion paper \cite{ABG10}. Their results have been generalized in several aspects. Within a Brownian setting, Bhamidi et. al. \cite{BSW17} provided scaling limits of measured metric spaces for a large class of inhomogeneous graph models. Continuum limits related to excursions of ``L\'{e}vy processes without replacement"  are described in \cite{BHS18}. In an $\alpha$-stable setting, continuum limits are described in the work of  Conchon-Kerjan and Goldschmidt \cite{CKG20}, where the limiting objects are related to \textit{tilting} excursions of a spectrally positive $\alpha$-stable process and their corresponding height processes (cf. \cite{LL98a,LL98b}). See also, \cite{GHS18} for more information about the continuum limits in the $\alpha$-stable setting.

\subsection{Epidemic Models}\label{sec:RFER}

The Reed-Frost model of epidemics describes the spread of a disease in a population of $n$ individuals in discrete time. It is described in terms of two processes $I = (I(t); t = 0,1,\dotsm)$ and $S = (S(t); t =0,1,\dotsm)$ where $I$ represents the number of infected individuals and $S$ represents the number of susceptible individuals. At each time $t$, every infected individuals as a probability $p$ of coming in contact with a susceptible individual and infecting that individual. 

It is further assumed that each infected individual at time $t$ recovers at time $t+1$. While $I$ itself is not Markov since the number of people who can be infected at time $t+1$ depend on the total number of infected individuals by time $t$, the pair $(S,I)$ is. Moreover, it can be easily seen that 
\begin{equation*}
\left(I(t+1) \big| I(t) = i,  S(t) =s \right) \overset{d}{=} \text{Bin}\left(s, 1-(1-p)^i \right), 
\end{equation*} and $S(t+1)$ is obtained by setting $S(t)-I(t+1)$. 

As described in \cite{BM90}, the Reed-Frost model can be described as exploring an \ER graph $G(n,p)$. We quote them at length:
\begin{quote}
	[O]ne or more initial vertices [of $G(n,p)$] are chosen at random as the $I(0)$ initial infectives, their neighbors become the $I(1)$ infectives at time $1$, and, inductively the $I(t+1)$ infectives at time $t+1$ are taken to be those neighbours of the $I(t)$ infectives at time $t$ which have not previously been infected.
\end{quote}

In \cite{vBML80}, von Bahr and Martin-L\"{o}f give a back-of-the-envelope calculation to show that the Reed-Frost epidemic should have a scaling limit for the process $I$ when $p = (n^{-1} + \lambda n^{-4/3}+o(n^{-4/3}))$ and suitable Lindeberg conditions hold. In Lemma \ref{lem:asympt}, we provide detailed results on the asymptotics in this regime, and generalize this with Lemma \ref{lem:asympt2} to the regime where $p = n^{-1}+\lambda \theta_n n^{-4/3}$ where $\theta_n\to \infty$, but $\theta_n = o(n^{1/3})$. 
\subsection{The breadth-first labeling} \label{sec:labeling}

The breadth-first labeling we will use can be described on any graph, so we will describe it on a generic finite graph $G$ with $n$ vertices. See Figure \ref{fig:comp} below as well. The breadth-first ordering in this work is described, in short, as follows
\begin{itemize}
	\item Randomly select $k$ distinct vertices in a graph $G$. Call these vertices the roots and label these by $\rho(j)$ for $j = 1,2\dotsm$.
	\item Begin the labeling of these roots by $w^k(0) = \rho(1)$, $w^k(1) = \rho(2),\dotsm w^k(k-1) = \rho(k)$,
	\item Label the unlabeled vertices neighboring vertex $w^k(0)$ by $w^k(k),w^k(k+1),\dotsm,w^k(\ell-1)$. 
	\item After exploring the neighbors of vertex $w^k(j-1)$ and using the labels $w^k(0),\dotsm, w^k(m-1)$ (say), label the unlabeled vertices neighboring vertex $w^k(j)$ by $w^k(m),w^k(m+1), \dotsm$. Continue this until all labeled vertices have been explored (which occurs when you explore every vertex connected to $\{\rho(1),\rho(2),\dotsm, \rho(k)\}$).
	
\end{itemize}
In more detail we label all the roots as in the second bullet point above as $w^k(j)$ for $j = 0,1,\dotsm, k-1$, and define the vertex set $\mathscr{V}^k(1) = \{w^k(0),\dotsm, w^k(k-1)\}.$ Now given a vertex set $\mathscr{V}^k(i) = \{w^k(j),w^k(j+1),\dotsm, w^k(\ell)\}$ we define the vertices $w^k(\ell+1), w^k(\ell+2),\dotsm,w^k(\ell+c)$ as the unlabeled vertices which are a neighbor of $w^k(j)$ (if any), where $c$ represents the total number of such vertices. Then define the vertex set $$\mathscr{V}^k(i+1) = \mathscr{V}^k(i)\setminus\{ w^k(j) \}\cup\{w^k(\ell+1),\dotsm,w^k(\ell+c) \}.$$

When we have a sequence of graphs $(G_n; n = 1,2,\dotsm)$, we include a subscript $n$ for both the roots and the breadth-first labeling. That is we write $\rho_n(j)$ and $w_n^k(j)$. 

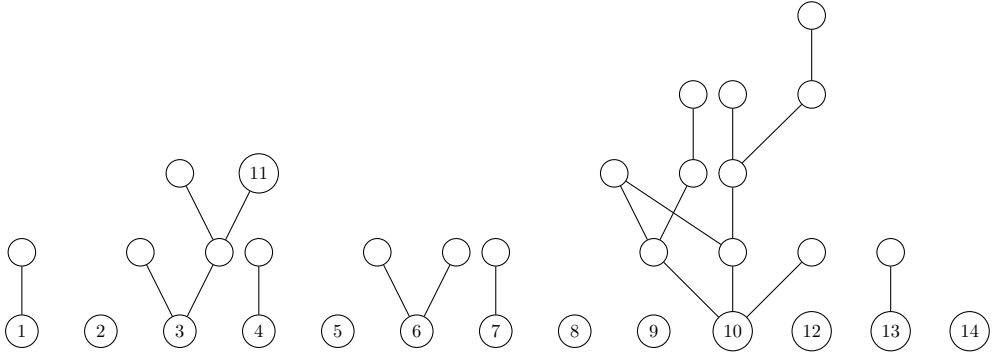
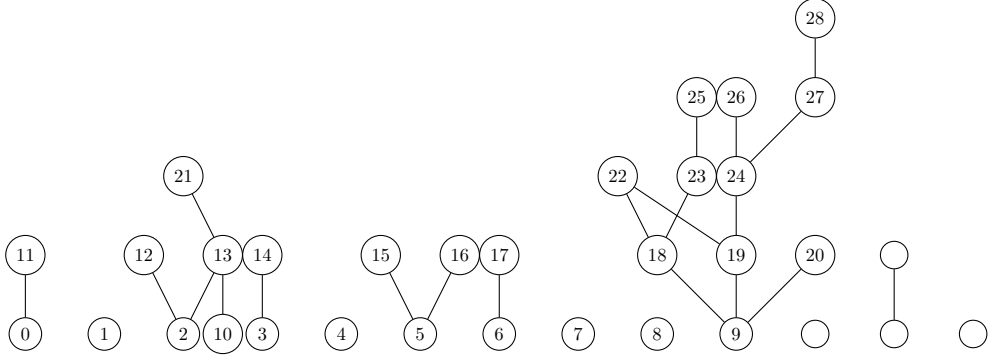
\begin{figure}[h!] 
	\begin{subfigure}[t]{\textwidth}
		\begin{tikzpicture}[scale=.7]
		
		\node (parent) {} [grow'=up]
		child {
			node [circle,scale=.7,draw] (a) {1} edge from parent[draw=none]
			child {
				node [circle,scale=.7,draw] (b) {\phantom{-}}
			}
		}
		child {
			node [circle,scale=.7,draw] (c) {2} edge from parent[draw=none]
		}
		child {
			node [circle,scale=.7, draw] (d) {3} edge from parent[draw=none]
			child{
				node [circle,scale=.7,draw] (e) {\phantom{-}}
			}
			child{
				node [circle,scale=.7,draw] (f) {\phantom{-}}
				child{
					node [circle,scale=.7,draw] (g) {\phantom{-}}
				}
				child{
					node [circle,scale=.7,draw] (h) {11}
				}
			}
		}
		child {
			node [circle,scale=.7,draw] (i) {4} edge from parent[draw=none]
			child{
				node [circle,scale=.7,draw] (j) {\phantom{-}}
			}
		}
		child {
			node [circle,scale=.7,draw] (k) {5} edge from parent[draw=none]
		}
		child{
			node [circle,scale=.7,draw] (l) {6} edge from parent[draw=none]
			child{
				node [circle,scale=.7,draw] (m) {\phantom{-}}
			}
			child{
				node [circle,scale=.7,draw] (n) {\phantom{-}}
			}
		}
		child{
			node [circle,scale=.7,draw] (o) {7} edge from parent[draw=none]
			child{
				node [circle,scale=.7,draw] (p) {\phantom{-}}
			}
		}
		child{
			node[circle,scale=.7,draw] (q) {8} edge from parent[draw=none]
		}
		child {
			node[circle,scale=.7,draw] (r) {9} edge from parent[draw=none]
		}
		child {
			node[circle,scale=.7,draw] (s) {10} edge from parent[draw=none]
			child {
				node[circle,scale=.7,draw] (t) {\phantom{-}}
				child {
					node[circle,scale=.7,draw] (u) {\phantom{-}} 
				} 
				child {
					node[circle,scale=.7,draw] (v) {\phantom{-}}
					child {
						node[circle,scale=.7,draw] (w) {\phantom{-}} 
					} 
				}
			}
			child {
				node[circle,scale=.7,draw] (x) {\phantom{-}}
				child {
					node[circle,scale=.7,draw] (y) {\phantom{-}} 
					child {
						node (blank2) {} edge from parent[draw=none]
					}
					child {
						node[circle,scale=.7,draw] (z) {\phantom{-}} 
					}
					child {
						node[circle,scale=.7,draw] (aa) {\phantom{-}} 
						child {
							node[circle,scale=.7,draw] (ab) {\phantom{-}} 
						}
					}
				} 
			}
			child {
				node[circle,scale=.7,draw] (ac) {\phantom{-}} 
			}	
		}
		child {
			node[circle,scale=.7,draw] (ad) {12} edge from parent [draw=none]
		}
		child {
			node[circle,scale=.7,draw] (ae) {13} edge from parent [draw=none]
			child {
				node[circle,scale=.7,draw] (af) {\phantom{-}} 
			} 
		}
		child {
			node[circle,scale=.7,draw] (ag) {14} edge from parent [draw=none]
		}
		;
		\draw[-] (x) -- (u);
		\end{tikzpicture}	
		\subcaption{The example graph used in Aldous \cite[Fig. 1]{A97_1}. The labeled vertices are $\rho_n(1),\dotsm, \rho_n(14)$.}
	\end{subfigure}
	\newline

	\begin{subfigure}[b]{\textwidth}
		\begin{tikzpicture}[scale=.7]
		\node (parent) {} [grow'=up]
		child {
			node [circle,scale=.7,draw] (a) {0} edge from parent[draw=none]
			child {
				node [circle,scale=.7,draw] (b) {11}
			}
		}
		child {
			node [circle,scale=.7,draw] (c) {1} edge from parent[draw=none]
		}
		child {
			node [circle,scale=.7, draw] (d) {2} edge from parent[draw=none]
			child{
				node [circle,scale=.7,draw] (e) {12}
			}
			child{
				node [circle,scale=.7,draw] (f) {13}
				child{
					node [circle,scale=.7,draw] (g) {21}
				}
				child [grow=down]{
					node [circle,scale=.7,draw] (h) {10}
				}
			}
		}
		child {
			node [circle,scale=.7,draw] (i) {3} edge from parent[draw=none]
			child{
				node [circle,scale=.7,draw] (j) {14}
			}
		}
		child {
			node [circle,scale=.7,draw] (k) {4} edge from parent[draw=none]
		}
		child{
			node [circle,scale=.7,draw] (l) {5} edge from parent[draw=none]
			child{
				node [circle,scale=.7,draw] (m) {15}
			}
			child{
				node [circle,scale=.7,draw] (n) {16}
			}
		}
		child{
			node [circle,scale=.7,draw] (o) {6} edge from parent[draw=none]
			child{
				node [circle,scale=.7,draw] (p) {17}
			}
		}
		child{
			node[circle,scale=.7,draw] (q) {7} edge from parent[draw=none]
		}
		child {
			node[circle,scale=.7,draw] (r) {8} edge from parent[draw=none]
		}
		child {
			node[circle,scale=.7,draw] (s) {9} edge from parent[draw=none]
			child {
				node[circle,scale=.7,draw] (t) {18}
				child {
					node[circle,scale=.7,draw] (u) {22} 
				} 
				child {
					node[circle,scale=.7,draw] (v) {23}
					child {
						node[circle,scale=.7,draw] (w) {25} 
					} 
				}
			}
			child {
				node[circle,scale=.7,draw] (x) {19}
				child {
					node[circle,scale=.7,draw] (y) {24} 
					child {
						node (blank2) {} edge from parent[draw=none]
					}
					child {
						node[circle,scale=.7,draw] (z) {26} 
					}
					child {
						node[circle,scale=.7,draw] (aa) {27} 
						child {
							node[circle,scale=.7,draw] (ab) {28} 
						}
					}
				} 
			}
			child {
				node[circle,scale=.7,draw] (ac) {20} 
			}	
		}
		child {
			node[circle,scale=.7,draw] (ad) {\phantom{-}} edge from parent [draw=none]
		}
		child {
			node[circle,scale=.7,draw] (ae) {\phantom{-}} edge from parent [draw=none]
			child {
				node[circle,scale=.7,draw] (af) {\phantom{-}} 
			} 
		}
		child {
			node[circle,scale=.7,draw] (ag) {\phantom{-}} edge from parent [draw=none]
		}
		;
		\draw[-] (x) -- (u);
		\end{tikzpicture}
		\subcaption{The breadth-first labelling of the example graph in \cite{A97_1} where $\rho_n(j) = j$ for $j=1,2,\dotsm,11$. In lieu of writing $w_n^{11}(j)$ in each vertex, we just write $j$ instead. }	
	\end{subfigure}
	\caption{Exploration of the graph.}	\label{fig:comp}
\end{figure}

\section{Limit of Height Profile}\label{sec:z}

In this section we provide lemmas necessary to go from the convergence in \cite{DL06, Simatos15} of a continuous-time epidemic model to the statement presented in Lemma \ref{thm:zconv}. This is hinted at in \cite[Appendix 2]{vBML80} as well. We fix a $\lambda\in \R$ and let $\G_n = G(n,n^{-1}+\lambda n^{-4/3})$ denote an \ER random graph, and let $Z_n^k$ be defined by \eqref{eqn:zDiscrete}. For convenience, we let $C_n^k = \left(C_n^k(h); h = 0,1,\dotsm \right)$ be defined by 
\begin{equation}\label{eqn:cDiscDef}
C_n^k(h) = \sum_{j=0}^h Z_n^k(j).
\end{equation} In terms of the graph $\G_n$, $C_n^k(h)$ represents the number of vertices within distance $h$ of the $k$ randomly selected vertices $\{\rho_n(1),\dotsm, \rho_n(k)\}$. In terms of the SIR model, $C_n^k(h)$ represents the number of individuals who have contracted the disease at or before ``time" $h$.

From the correspondence of the Reed-Frost model and the \ER random graph, we know that $(Z_n^k(h),C_n^k(h))$ is a Markov chain with state space
\begin{equation*}
\mathcal{S} = \{(z,c)\in \mathbb{Z}^2: z,c\ge 0\}
\end{equation*} which is absorbed upon hitting the line $\{(0,c): c \ge 0\}$. Moreover, the conditional distribution of $Z_n^k(h+1)$ given $\left(Z_n^k(h),C_n^k(h)\right)$ is
\begin{equation}\label{eqn:condDist}
\left(Z_n^k(h+1) \bigg| Z_n^k(h) = z, C_n^k(h) = c\right) \overset{d}{=} \left\{ 
\begin{array}{ll}
\text{Bin}(n-c, q(n,z))&:z>0, c<n\\
0&: \text{else}
\end{array}
\right.,
\end{equation} where $q(n,z)$ is defined by
\begin{equation}\label{eqn:qDef}
q(n,z) = 1- \left(1-n^{-1} - \lambda n^{-4/3} \right)^z.
\end{equation} The joint conditional distribution of $(Z_n^k(h+1),C_n^k(h+1))$ is easily deduced from equations \eqref{eqn:condDist} and \eqref{eqn:cDiscDef}.

\subsection{Asymptotics for binomial statistics}

We begin by examining the binomial random variables
$$
\beta(n,z,c) \overset{d}{\sim} \text{Bin}(n-c,q(n,z)),
$$
where $q(n,z)$ is defined by \eqref{eqn:qDef}. Examining the convergence in Theorem \ref{thm:zconv}, we'll study various statistics of $\beta(n,z,c)$ as $n\to\infty$ with $z = O(n^{1/3})$ and $c = O(n^{2/3})$. 

We define the following statistics
\begin{equation}\label{eqn:betaStats}
\mu(n,z,c) = \E\left[\beta(n,z,c) \right],\quad \sigma^2(n,z,c) = \Var\left[\beta(n,z,c) \right],
\quad \kappa(n,z,c) = \E\left[(\beta(n,z,c) - z)^4 \right].
\end{equation}

The main purpose of this subsection is to establish the following lemma:
\begin{lem}\label{lem:asympt}
Fix an $r>0$ and $T>0$ and let
\begin{equation*}
\Omega_n = \Omega (n,r,T) = \left\{(z,c)\in \mathbb{Z}^2: 0\le z\le n^{1/3}r, 0\le c\le n^{2/3} Tr \right\}.
\end{equation*}
Then, as $n\to\infty$, the following bounds hold:
\begin{equation}\label{eqn:asypmtotics}
\begin{split}
&\sup_{\Omega_n} \left|\mu(n,z,c) - z - n^{-1/3} z(\lambda - n^{-2/3}c) \right| = O(n^{-1/3})\\
&\sup_{\Omega_n} \left| \sigma^2(n,z,c) - z - n^{-1/3} z(\lambda - n^{-2/3} c) \right| = O(n^{-1/3})\\
&\sup_{\Omega_n} \left|\kappa(n,z,c)\right| = O(n^{2/3})
\end{split}
\end{equation} In particular, \begin{equation*}
\begin{split}
&\sup_{\Omega_n} \left| \mu(n,z,c) - z\right| = O(1)\\
&\sup_{\Omega_n} \left| \sigma^2(n,z,c) - z\right| = O(1)
\end{split}, \qquad \text{as }n\to\infty.
\end{equation*}
\end{lem}

\begin{proof}
We prove the statements in equation \eqref{eqn:asypmtotics}, since the latter bounds easily follow from the more detailed asymptotics.

We start with the expansion of $\mu(n,z,c)$. The binomial theorem gives
\begin{align*}
\mu(n,z,c) &= (n-c) \left(1-(1-n^{-1}-\lambda n^{-4/3})^z \right)\\
&= (n-c) \left(z(n^{-1}+\lambda n^{-4/3}) - \sum_{j=2}^z \binom{z}{j} (-1)^j (n^{-1} + \lambda n^{-4/3})^j \right)\\
&= z + n^{-1/3} z(\lambda -n^{-2/3}c) - \lambda n^{-4/3} z c - (n-c)\sum_{j=2}^z \binom{z}{j} (-1)^j (n^{-1}+\lambda n^{-4/3})^j.
\end{align*}
For $n$ sufficiently large, we can obtain the bounds
\begin{align*}
\left|\mu(n,z,c) - z - n^{-1/3}z(\lambda - n^{-2/3}c)  \right|&\le |\lambda| n^{-4/3} z c + (n-c)\left|\sum_{j=2}^z \binom{z}{j} (-1)^j (n^{-1}+\lambda n^{-4/3})^j\right|\\
 &\le |\lambda| n^{-4/3} zc + n\sum_{j=2}^z \binom{z}{j} (n^{-1}+\lambda n^{-4/3})^j\\
&\le |\lambda| n^{-4/3}zc + n\sum_{j=2}^z \left(\frac{2ez}{n}\right)^j.
\end{align*} In the second and third inequality above, we used the bound $0<n^{-1}+\lambda n^{-4/3} \le 2n^{-1}$ and the bound $\displaystyle \binom{m}{k} \le (em)^k.$

Taking the supremum over $\Omega_n$, gives
\begin{align*}
\sup_{\Omega_n} \left|\mu(n,z,c) - z - n^{-1/3}z(\lambda - n^{-2/3}) \right| &\le n^{-1/3} |\lambda| Tr^2 + n\sum_{j=2}^{n^{1/3} r} \left( \frac{2er}{n^{2/3}}\right)^j\\
&\le n^{-1/3} |\lambda|Tr^2 + n\left(\frac{4e^2r^2 n^{-4/3}}{1-2ern^{-2/3}}\right)\\
&\le \left(|\lambda|Tr^2 + 8e^2r^2 \right)n^{-1/3} = O(n^{-1/3}).
\end{align*}
This proves the desired expansion and bound for $\mu(n,z,c)$.

We now examine the bounds for $\sigma^2(n,z,c)$. Again, we use the binomial theorem
\begin{align}
\sigma^2(n,z,c) &= \mu(n,z,c) (1-n^{-1} - \lambda n^{-4/3})^z \nonumber\\
&= \mu(n,z,c) \left(1-z \left(n^{-1}+\lambda n^{-4/3}\right) + \sum_{j=2}^z \binom{z}{j} (-1)^j(n^{-1}+\lambda n^{-4/3})^j  \right)\nonumber\\
&= \mu(n,z,c) - \mu(n,z,c)z(n^{-1}+ \lambda n^{-4/3}) + \mu(n,z,c)\sum_{j=2}^z \binom{z}{j} (-1)^j (n^{-1}+\lambda n^{-4/3})^j. \label{eqn:sigmaBound}
\end{align} We can then use the previous asymptotic bounds for $\mu(n,z,c)$ to get
\begin{equation*}
\sup_{\Omega_n} \left|\sigma^2(n,z,c) - z - n^{-1/3}z(\lambda -n^{-2/3}) \right| \le \sup_{\Omega_n} |\sigma^2(n,z,c) - \mu(n,z,c)|+O(n^{-1/3}).
\end{equation*} We can bound the first term on the right-hand side as we did for $\mu(n,z,c)$ above. We get
\begin{align*}
\sup_{\Omega_n}& \left|\sigma^2(n,z,c) - \mu(n,z,c) \right| \le \sup_{\Omega_n} \left|\mu(n,z,c)z (n^{-1}+\lambda n^{-4/3}) + n\sum_{j=2}^z \binom{z}{j} (n^{-1}+\lambda n^{-4/3})^j \right| \\
&\le \sup_{\Omega_n} \left(2rn^{-2/3} \mu(n,z,c) + n \sum_{j=2}^z \binom{z}{j} (n^{-1} +\lambda n^{-4/3})^j \right)\\
&\le 2rn^{-2/3} \sup_{\Omega_n}\left(|\mu(n,z,c) - z - n^{-1/3}z(\lambda - n^{-2/3}c)| + |z+n^{-1/3}z(\lambda- n^{-2/3}c)| \right)\\
&\qquad + O(n^{-1/3})\\
&= 2rn^{-2/3} \left(O(n^{-1/3})+ O(n^{1/3}) \right) + O(n^{-1/3})\\
&= O(n^{-1/3}).
\end{align*} In the first line we used the bound $0<n^{-1}+\lambda n^{-4/3} \le 2n^{-1}$ for large enough $n$, and $\mu(n,z,c)\le n$, the second inequality used the bounds $(n^{-1} + \lambda n^{-4/3})z\le 2rn^{-2/3}$ on $\Omega_n$. The third inequality used the previously derived bound of $\displaystyle \sum_{j=2}^z \binom{z}{j}(n^{-1}+\lambda n^{-4/3}) \le 8e^2 r^2 n^{-4/3}$ which holds for $n$ sufficiently large.

To show the bound for $\kappa(n,z,c)$, we expand it as follows
\begin{align*}
\kappa(n,z,c) &= \E\left[(\beta(n,z,c)-\mu(n,z,c))^4 \right] + 4 \E\left[(\beta(n,z,c)-\mu(n,z,c))^3\right](\mu(n,z,c)-z) \\
&\qquad + 6 \E\left[ (\beta(n,z,c)-\mu(n,z,c))^2 \right](\mu(n,z,c)-z)^2  \\
&\qquad + 4\E\left[\beta(n,z,c)-\mu(n,z,c)\right](\mu(n,z,c)-z)^3\\
&\qquad +(\mu(n,z,c)-z)^4\\
&=: \kappa_4(n,z,c) + 4\kappa_3(n,z,c) + 6 \kappa_2(n,z,c) +0 + \kappa_0(n,z,c).
\end{align*}
We now show that $\kappa_j(n,z,c)$ for $j = 0,2,3,4$ have the desired bound.

By the approximations for $\mu(n,z,c)$ it is easy to see that $$
\sup_{\Omega_n}|\kappa_0(n,z,c)| = O(1),\qquad \text{as }n\to\infty.
$$ Similarly, we can use the approximations for both $\mu(n,z,c)$ and $\sigma^2(n,z,c)$ to arrive at
\begin{align*}
\sup_{\Omega_n}| \kappa_2(n,z,c) |&= \sup_{\Omega_n} \left|\sigma^2(n,z,c)(\mu(n,z,c) - z)^2\right|\\
&\le O(1) \cdot \sup_{\Omega_n} \left(\left| \sigma^2(n,z,c)-z- n^{-1/3}z(\lambda -n^{-2/3}c)\right|+ |z+n^{-1/3}z(\lambda - n^{-2/3}c)|\right)\\
&= O(1) \cdot \left(O(n^{-1/3}) + O(n^{1/3})\right) = O(n^{1/3}).
\end{align*}

Using the third central moment of a binomial random variable gives
\begin{align*}
\kappa_3(n,z,c) = \sigma^2(n,z,c)(1-2q(n,z)) (\mu(n,z,c)-z).
\end{align*} A similar expansion as that for $\kappa_2(n,z,c)$ shows that $\sup_{\Omega_n}|\sigma^2(n,z,c)| = O(n^{1/3})$, and the other two terms are $O(1)$ over $\Omega_n$ and hence
\begin{equation*}
\sup_{\Omega_n} |\kappa_3(n,z,c)| = O(n^{1/3}).
\end{equation*}

By the fourth central moment for a binomial random variable, we have
\begin{align*}
|\kappa_4(n,z,c)| &= \sigma^2(n,z,c)\left|1 + 3(n-2-c)(q(n,z)-q(n,z)^2)\right|\\
&\le \sigma^2(n,z,c) \left(\left|1 + 3(n-c)(q(n,z)-q(n,z)^2) \right|+2|(q(n,z)-q(n,z)^2) | \right)\\
&\le \sigma^2(n,z,c) \left(3 + 3\sigma^2(n,z,c) \right).
\end{align*} Hence, by the bound for $\sigma^2(n,z,c)$
\begin{equation*}
\sup_{\Omega_n} |\kappa_4(n,z,c)| = O(n^{2/3}).
\end{equation*} This proves the desired claim.

\end{proof}

\subsection{Martingale estimates}

In this section we verify the conditions of the martingale functional central limit theorem, as found in \cite[Theorem 7.4.1]{EK86}. Before moving onto the lemma, we establish some notation.

We let 
\begin{equation*}
\F_n^k(h) = \sigma\left( Z_n^k(j) : j\le h\right)
\end{equation*} denote the filtration generated by $Z_n^k$. We let 
\begin{equation*}
Z_n^k(h) = k + M_n^k(h) + B_n^k(h),
\end{equation*}
be the Doob decomposition of $Z_n^k$ into an $\{\F_n^k(h)\}_{h\ge 0}$-martingale $M_n^k$ and a predictable process $B_n^k$. Similarly, we let $Q_n^k$ be the unique increasing process which makes $(M_n^k(h))^2-Q_n^k(h)$ an $\{\F_n^k(h)\}_{h\ge 0}$-martingale. That is
\begin{equation*}
\begin{split}
B_n^k(h) &= \sum_{\ell = 0}^{h-1} \E\left[ Z_n^k(\ell+1) - Z_n^k(\ell) | \F_n^k(\ell)\right]\\
Q_n^k(h) &= \sum_{\ell=0}^{h-1} \E\left[\left(Z_n^k(\ell+1)-Z_n^k(\ell) \right)^2\big| \F_n^k(\ell) \right] - \E\left[Z_n^k(\ell+1)-Z_n^k(\ell) | \F_n^k(\ell) \right]^2.
\end{split}
\end{equation*}

We define the following rescaled processes
\begin{equation}\label{eqn:rescales} 
\begin{split}
\tilde{Z}_n^k(t) &= n^{-1/3}Z_n^k(\fl{n^{1/3}t})\qquad \tilde{C}_n^k(t) = n^{-2/3}C_n^k(\fl{n^{1/3}t})\qquad \\
\tilde{M}_n^k(t) &= n^{-1/3}M^k_n(\fl{n^{1/3}t})\qquad \tilde{B}_n^k(t) = n^{-1/3}B^k_n(\fl{n^{1/3}t}) \qquad \tilde{Q}_n^k(t)= n^{-2/3}Q^k_n(\fl{n^{1/3}t})
\end{split}.
\end{equation} Also define $\tau_n^k(r) = \inf\{t:\tilde{Z}_n^k(t)\vee \tilde{Z}_n^k(t-)\ge r\}$ and $\hat{\tau}_n^k(r) = n^{-1/3}\inf\{h: Z_n^k(h)\ge n^{1/3}r\}$.

\begin{lem} \label{lem:ekVerify} Fix any $r>0$, $T>0$ and $x>0$. Let $k = k(n) = \fl{n^{1/3}x}$. 
	The following limits hold
	\begin{enumerate}
		\item $\displaystyle \lim_{n\to\infty}\E\left[ \sup_{t\le T\wedge \tau_n^k(r)} |\tilde{Z}_n^k(t)-\tilde{Z}_n^k(t-)|^2\right]  = 0$.
		\item $\displaystyle \lim_{n\to\infty}\E\left[ \sup_{t\le T\wedge \tau_n^k(r)} |\tilde{B}_n^k(t)-\tilde{B}_n^k(t-)|^2\right]  = 0$.
		\item $\displaystyle \lim_{n\to\infty}\E\left[ \sup_{t\le T\wedge \tau_n^k(r)} |\tilde{Q}_n^k(t)-\tilde{Q}_n^k(t-)| \right] = 0$.
		\item $\displaystyle \sup_{t\le T\wedge \tau_n^k(r)} \left|\tilde{Q}_n^k(t)  - \int_0^t \tilde{Z}^k_n(s)\,ds \right| \longrightarrow 0$, in probability as $n\to\infty$.
		\item $\displaystyle  \sup_{t\le T\wedge \tau_n^k(r)}\left| \tilde{B}_n^k(t)- \int_0^t (\lambda -\tilde{C}^k_n(s))\tilde{Z}^k_n(s) \,ds\right|{\longrightarrow}0$, in probability as $n\to\infty$.
	\end{enumerate}
\end{lem}

\begin{proof}
	
	In order to show (1), we prove the stronger claim
	\begin{equation*}
	\lim_{n\to\infty}\E\left[ \sup_{t\le T\wedge \tau_{n}^k(r)} |\tilde{Z}_n^k(t)-\tilde{Z}_n^k(t-)|^4\right] = 0.
	\end{equation*}
	To show this, note the following string of inequalities
	\begin{align*}
	\E&\left[\sup_{t\le T\wedge \tau_n^k(r)} |\tilde{Z}_n^k(t)-\tilde{Z}_n^k(t-)|^4 \right] \le n^{-4/3} \E\left[\sup_{h\le n^{1/3}(T\wedge \hat{\tau}_n^k(r))} |Z_n^k(h+1)-Z_n^k(h)|^4 \right]\\
	&\le n^{-4/3} \sum_{h=0}^{\fl{n^{1/3}(T\wedge \hat{\tau}_n^k(r))}} \E\left[|Z_n^k(h+1)-Z_n^k(h)|^4 \right]\\
	&\le n^{-4/3} \sum_{h=0}^{\fl{n^{1/3}(T\wedge \hat{\tau}^k_n(r))}}\sup_{\Omega_n} \E \left[\E\left[(Z_n^k(h+1)-Z_n^k(h))^4\bigg| Z_n^k(h) = z, C_n^k(h) = c\right] \right]\\
	&\le Tn^{-1} \sup_{\Omega_n} \E\left[ (\beta(n,z,c)-z)^4 \right] = O(n^{-1/3}).
	\end{align*} In the third inequality above, we used the tower property and on fact that for $h\le n^{1/3}(T\wedge \hat{\tau}^k_n(r))$ both $Z_n^k(h)\le n^{1/3}r$ and $C_n^k(h)\le n^{2/3}Tr$. The fourth inequality used the Markov property of $(Z_n^k,C_n^k)$. The convergence then holds by the asymptotic result for $\kappa(n,z,c)$ shown in Lemma \ref{lem:asympt}

	To verify (2), we begin by noting that 
	\begin{equation*}
	B_n^k(h) = \sum_{j=0}^{h-1} \E\left[(Z_n^k(j+1)-Z_n^k(j))|\F_n^k(j) \right],
	\end{equation*} and hence
	\begin{equation*}
	\sup_{t\le T\wedge \tau_n^k(r)} |\tilde{B}_n^k(t)-\tilde{B}_n^k(t-)|^2 \le n^{-2/3} \sup_{h\le n^{1/3}(T\wedge \hat{\tau}_n^k(r))} \left|\E\left[Z_n^k(h+1)-Z_n^k(h)\Big|\F_n^k(h) \right] \right|^2.
	\end{equation*}
	We also note that almost surely on $h\le n^{1/3}(T\wedge \hat{\tau}^k_n(r))$
	\begin{equation*}
	\E\left[ Z_n^k(h+1) - Z_n^k(h) \Big|\F_n^k(h) \right]\le \sup_{\Omega_n} |\E[(\beta(n,z,c) - z)]| = O(1)\text{ as }n\to\infty, 
	\end{equation*} by Lemma \ref{lem:asympt}. Hence,
	\begin{align*}
	\E \left[\sup_{t\le T\wedge \tau_n^k(r)} |\tilde{B}_n^k(t)-\tilde{B}_n^k(t-)|^2 \right]&\le \E \left[ n^{-2/3} \sup_{h\le n^{1/3}(T\wedge \hat{\tau}^k_n(r))} \left|\E\left[Z_n^k(h+1)-Z_n^k(h)\Big|\F_n^k(h) \right] \right|^2 \right]\\
	&\le n^{-2/3} \cdot O(1) = O(n^{-2/3}),
	\end{align*} which argues (2).
	
	We next show (3). We begin by noting that
	$$
	Q_n^k(h) = \sum_{j=0}^{h-1} \E\left[ \left(Z_n^k(j+1) - Z_n^k(j) \right)^2\Big| \F_n^k(j)\right] - \E\left[Z_n^k(j+1)-Z_n^k(j) \Big| \F_n^k(j)\right]^2.
	$$ Hence
	\begin{align*}
	\E&\left[\sup_{t\le T\wedge \tau_n^k(r)} |\tilde{Q}_n^k(t)-\tilde{Q}_n^k(t-)|\right] \\
	&\le n^{-2/3} \E \left[\sup_{h\le n^{1/3}(T\wedge \hat{\tau}^k_n(r))} \E[(Z_n^k(h+1)-Z_n^k(h))^2 |\F_n^k(h)]+ \E[Z_n^k(h+1)-Z_n^k(h)|\F_n^k(h)]^2 \right]\\
	&\le n^{-2/3} \E\left[\sup_{h\le n^{1/3}(T\wedge \hat{\tau}^k_n(r))} \sup_{\Omega_n} \left|\E[(\beta(n,z,c)-z)^2] + (\mu(n,z,c)-z)^2\right|\right]\\
	&=n^{-2/3} \sup_{\Omega_n} \left|\sigma^2(n,z,c) + 2(\mu(n,z,c)-z)^2  \right| = O(n^{-1/3}).
	\end{align*} In the last equalities, we Lemma \ref{lem:asympt} and the observation that $\sup_{\Omega_n} \sigma^2(n,z,c) = O(n^{1/3})$.
	
	To argue claim (4), we observe
	\begin{align*}
	\left(Q_n^k(h+1)-Q_n^k(h) \right) &= \E\left[(\beta(n,Z_n^k(h),C_n^k(h)) - Z_n^k(h))^2 |\F_n^k(h) \right] \\ &\qquad\qquad  - (\mu(n,Z_n^k(h),C_n^k(h)) - Z_n^k(h))^2\\
	& = \sigma^2(n,Z_n^k(h),C_n^k(h)).
	\end{align*} 
	Therefore,
	\begin{align*}
	\sup_{h\le n^{1/3}(T\wedge \hat{\tau}_n^k(r))}\left|Q_n^k(h) - \sum_{j=0}^{h-1} Z_n^k(j) \right| &=\sup_{h\le n^{1/3}(T\wedge \hat{\tau}^k_n(r))} \left|\sum_{j=0}^{k-1}  \sigma^2(n,Z_n^k(j),C_n^k(j)) - Z_n^k(j)\right|\\
	&\le \sup_{h\le n^{1/3}(T\wedge \hat{\tau}^k_n(r))}\sum_{j=0}^{h-1} |\sigma^2(n,Z_n^k(j),C_n^k(j))- Z_n^k(j)|\\
	&\le \sum_{j=0}^{n^{1/3}T} \sup_{\Omega_n} |\sigma^2(n,z,c) - z|\\
	&= O(n^{1/3}).
	\end{align*} In the third inequality above, we used the previously discussed bounds on $Z_n^k(h)$ and $C_n^k(h)$ for all $h$ such that $h\le n^{1/3}(T \wedge \hat{\tau}^k_n(r))$ and in the last term, we used the bound for $\sigma^2(n,z,c)-z$ on $\Omega_n$ given by Lemma \ref{lem:asympt}.
	Hence
	\begin{align}\label{eqn:expandSupQtilde}
	\sup_{t\le T\wedge \tau_{n}^k(r)} \left| \tilde{Q}_n^k(t) - \int_0^t \tilde{Z}_n^k(s)\,ds\right| &\le \sup_{t\le T\wedge \tau_n^k(r)} \left| \tilde{Q}_n^k(t) - n^{-2/3}\sum_{j=0}^{\fl{n^{1/3} t}} Z_n^k(j) \right|\\
	&\qquad \qquad +\sup_{t\le T\wedge \tau_n^k(r)} \left| n^{-2/3}\sum_{j=0}^{\fl{n^{1/3} t}} Z_n^k(j) - \int_0^t \tilde{Z}_n^k(s)\,ds\right| .\nonumber
	\end{align} We can bound the first term, using the bound for $Q_n^k(h)-\sum_{j=0}^{h-1}Z_n^k(j)$ from above, to get
	\begin{align*}
	\sup_{t\le T\wedge \tau_n^k(r)} \left| \tilde{Q}_n^k(t) - n^{-2/3}\sum_{j=0}^{\fl{n^{1/3} t}} Z_n^k(j) \right|& \le n^{-2/3} \sup_{h\le n^{1/3}(T\wedge \hat{\tau}_n(r))} \left|Q_n^k(h) - \sum_{j=0}^{h-1} Z_n^k(j) \right|\\&=O(n^{-1/3}).
	\end{align*}
	We can bound the second term as follows
	\begin{align*}
	\sup_{t\le T\wedge \tau_n^k(r)} &\left|n^{-2/3} \sum_{j=0}^{\fl{n^{1/3}t}} Z_n^k(j) -\int_0^t \tilde{Z}_n^k(s)\,ds \right|\\
&\le \sup_{t\le T\wedge \tau_n^k(r)}  \left|n^{-2/3}\int_0^{\fl{n^{1/3}t}+1} Z_n^k(\fl{u})\,du - \int_0^t \tilde{Z}_n^k(s)\,ds \right|\\
&\le \sup_{t\le T\wedge \tau_n^k(r)} \left|\int_0^{n^{-1/3}(\fl{n^{1/3}t}+1)} \tilde{Z}_n^k(s)\,du -\int_0^t \tilde{Z}_n^k(s)\,ds \right|\\
&\le r\sup_{t\le T} \left|t-n^{-1/3}(\fl{n^{1/3}t}+1) \right|\longrightarrow 0.
	\end{align*} The above bounds hold almost surely. This proves the convergence in (4).
	
	We lastly establish (5). We begin by noting that
	$$
	B_n^k(h+1) -B_n^k(h) = \E\left[Z_n^k(h+1)-Z_n^k(h)|\F_n^k(h) \right] =  \mu(n,Z_n^k(h),C_n^k(h))-Z_n^k(h).
	$$ Hence,
	\begin{equation*}
	B_n^k(h) = \sum_{j=0}^{h-1} \mu(n,Z_n^k(j),C_n^k(j))-Z_n^k(j).
	\end{equation*} Therefore, almost surely we have
	\begin{align*}
	\sup_{h\le n^{1/3}(T\wedge \hat{\tau}^k_n(r))} &\left|B_n^k(h) - \sum_{j=0}^{h-1} n^{-1/3}Z_n^k(j) (\lambda -n^{-2/3}C_n^k(j)) \right| \\
	&= \sup_{h\le n^{1/3}(T\wedge \hat{\tau}^k_n(r))} \left| \sum_{j=0}^{h-1}( \mu(n,Z_n^k(j),C_n^k(j)) - Z_n^k(j)) -n^{-1/3}Z_n^k(j)(\lambda -n^{-1/3}C_n^k(j)) \right|\\
	&\le \sup_{h\le n^{1/3}(T\wedge \hat{\tau}^k_n(r))}\sum_{j=0}^{h-1} \left|\mu(n,Z_n^k(j),C_n^k(j)) - Z_n^k(j) -n^{-1/3}Z_n^k(j)(\lambda -n^{-1/3}C_n^k(j))\right|\\
	&\le Tn^{1/3} \sup_{\Omega_n} \left| \mu(n,z,c) - z - n^{-1/3}z (\lambda -n^{-2/3}c) \right| = O(1).
	\end{align*}
	
	Hence,
	\begin{align*}
	\sup_{t\le T\wedge \tau_n^k(r)}&\left| \tilde{B}_n^k(t)- \int_0^t (\lambda -\tilde{C}_n^k(s))\tilde{Z}_n^k(s) \,ds\right|\\
	&\le \sup_{h\le n^{1/3}(T\wedge \hat{\tau}_n^k(r))} \left|n^{-1/3}B_n^k(h) - n^{-1/3}\sum_{j=0}^{h-1} n^{-1/3}Z_n^k(j) (\lambda -n^{-2/3}C_n^k(j)) \right|\\
	&\qquad +\sup_{t\le T\wedge \tau_n^k(r)} \left|n^{-1/3}\sum_{j=0}^{\fl{n^{1/3}t}} n^{-1/3}Z_n^k(j) (\lambda -n^{-2/3}C_n^k(j))  - \int_0^t (\lambda - \tilde{C}_n^k(s))\tilde{Z}_n^k(s)\,ds\right|.
	\end{align*} By factoring out an $n^{-1/3}$ from the first term on the right-hand side, it is easy to see that term is $O(n^{-1/3})$ almost surely. Examining the second term on the right-hand side, we get almost surely
	\begin{align*}
	&\sup_{t\le T\wedge \tau_n^k(r)} \left|n^{-1/3}\sum_{j=0}^{\fl{n^{1/3}t}} n^{-1/3}Z_n^k(j) (\lambda -n^{-2/3}C_n^k(j))  - \int_0^t (\lambda - \tilde{C}_n^k(s))\tilde{Z}_n^k(s)\,ds\right|\\
	&\quad= \sup_{t\le T\wedge \tau_n^k(r)} \left|n^{-1/3}\int_0^{\fl{n^{1/3}t}+1} n^{-1/3}Z_n^k(\fl{u})(\lambda -n^{-2/3} C_n^k(\fl{u}))\,du - \int_0^t (\lambda-\tilde{C}_n^k(s))\tilde{Z}_n^k(s)\,ds \right|\\
	&\quad= \sup_{t\le T\wedge\tau_n^k(r)} \left|\int_0^{n^{-1/3}(\fl{n^{1/3}t}+1)}\tilde{Z}_n^k(s)(\lambda - \tilde{C}_n^k(s))\,ds  - \int_0^t \tilde{Z}_n^k(s)(\lambda - \tilde{C}_n^k(s))\,ds\right|\\
	&\quad= (|\lambda|+Tr)r \sup_{t\le T}\left|t-n^{-1/3}(\fl{n^{1/3}t}+1) \right| \to 0.
	\end{align*}
	
	 This proves the lemma.

\end{proof}

\subsection{Existence and uniqueness lemma, and a corollary} \label{sec:zthm}

Using the functional central limit machinery found in \cite[Chapter 7]{EK86}, it is not difficult to argue Lemma \ref{thm:zconv} from Lemma \ref{lem:ekVerify} along with the following existence and uniqueness lemma: \begin{lem}\label{lem:uniquenessLemma}
	Fix a $\lambda\in \R$ and an $x\ge 0$. 
	\begin{enumerate}
		\item There exists a unique strong solution to the following stochastic differential equation
		\begin{equation}\label{eqn:zeqn.2}
		\begin{split}
		d\Z(t) &= \sqrt{\Z(t)} dW(t) + \left(\lambda - \C(t) \right)\,\Z(t)\,dt,\qquad \Z(0) = x\\
		d\C(t) &= \Z(t)\,dt,\qquad \C(0) = 0,
		\end{split}
		\end{equation} which is absorbed upon $\Z$ hitting zero.
		\item Given a weak solution $(\Z,\C)$ to the equation \eqref{eqn:zeqn.2}, then on an enlarged probability space there exists a Brownian motion $B$ such that $(\Z,\C)$ solves
		\begin{equation}\label{eqn:ctsTC}
		\Z(t) = x + \X^\lambda \left( \C(t) \wedge T_{-x}\right),
		\end{equation} where $\X^\lambda(t) = B(t) + \lambda t -\frac{1}{2}t^2$.
		\item Given $\X^\lambda(t) = B(t) +\lambda t-\frac{1}{2}t^2$ for a Brownian motion $B$ there exists a path-wise unique solution $(\Z,\C)$ where $\C(t) = \int_0^t \Z(s)\,ds$ and such a solution is a weak solution to \eqref{eqn:zeqn.2}.
	\end{enumerate}
\end{lem}

\begin{remark}
	We observe that the SDE in equation \eqref{eqn:zeqn.2} does not have a $\frac{1}{2}$, while in the integrated form found in equation \eqref{eqn:zsde1} there is such a term. This is because$$
	\int_0^t \C(s)\,\Z(s) \,ds = \int_0^t \C(s)\,d\C(s) = \frac{1}{2} \C(t)^2.
	$$
\end{remark}
\begin{proof}
	The strong existence and uniqueness in first item follows from the Yamada-Watanabe theorem \cite[Theorem 1]{YW71}. The absorption upon $\Z$ is obvious by stopping $(\Z,\C)$ upon $\Z$ hitting zero and observing that this stopped process still solves \eqref{eqn:zeqn.2}.
	
	The path-wise existence found in the third item follows from known theorems on random time-changes. See, for example \cite[Chapter VI, Section 1]{EK86}, \cite{CLU09} or \cite[Section 2]{CPU13}. 
	
	Now suppose that $(\Z,\C)$ solves \eqref{eqn:zeqn.2} for a Brownian motion $W$. We observe that the quadratic variation of $\Z$ is given by
	$$
	\< \Z\>(t) = \int_0^t \Z(s)\,ds.
	$$ Define the process $M(t) = \int_0^t \sqrt{\Z(s)}\,dW(s)$. Define $V(t) = \inf\{s:\C(s)>t\}$ with the convention that $\inf\emptyset = \infty$.
	
	Hence, by the Dambis, Dubins-Schwarz theorem \cite[Chapter V, Theorem 1.7]{RY99}, on an enlarged probability space, there exists a Brownian motion $\tilde{B}$ such that process
	$$
	B(t) = \left\{\begin{array}{ll}
	M(V(t))&: t< \int_0^\infty \Z(s)\,ds\\
	M(\infty) + \tilde{B}_{t-\<M\>(\infty)} &: t\ge \int_0^\infty \Z(s)\,ds.
	\end{array}
	\right.
	$$ is a Brownian motion. 
	
	We then have for $t<\int_0^\infty \Z(s)\,ds$:
	\begin{align*}
	\Z(V(t)) &=x+ M(V(t)) + \int_0^{V(t)} (\lambda - \C(s)) \Z(s)\,ds\\
	&= x+ B(t) + \int_0^t \left(\lambda - s \right)\,ds\\
	&= x+\X^\lambda(t).
	\end{align*} Observe $t<\int_0^\infty \Z(s)\,ds$ occurs if and only if $\Z(V(t))>0$. Indeed, since $\Z$ is continuous, non-negative and absorbed upon reaching zero then by \cite[Lemma 0.4.8]{RY99} $t\le \C(V(t)) = \int_0^{V(t)}\,\Z(s)\,ds < \int_0^\infty \Z(s)\,ds$ if and only if $\int_{V(t)}^\infty \Z(s)\,ds> 0$ which occurs if and only if $\Z(V(t))>0$.

	Hence, we can rewrite the above string of equalities as 
	\begin{equation*}
	\Z(V(t)) = x +\X^\lambda (t \wedge T_{-x}), 
	\end{equation*} which now holds for all $t$. Indeed, if $t\ge \int_0^\infty \Z(s)\,ds = \int_0^\zeta \Z(s)\,ds$ then $\V(t) = \C(\zeta) = T_{-x}$. Since $V$ and $\C$ are two-sided inverses of each other prior to $V(t)=\infty$, the above equation implies equation \eqref{eqn:ctsTC} holds.
	
	Reversing the above steps, gives the implication in the third item.
\end{proof}

\begin{lem}\label{lem:compactsupport}
	Let $(\Z,\C)$ be a solution of \eqref{eqn:zeqn.2}. Then almost surely $$
	\zeta:=\inf\{t:\Z(t) = 0\} <\infty.
	$$
\end{lem}
\begin{proof}
	We let $\X^\lambda(t)$ be the process define in Lemma \ref{lem:uniquenessLemma}(2) and such that $(\Z,\C)$ solves \eqref{eqn:ctsTC}. We can write $\C$ as a function of just the process $\X^\lambda$. Indeed let $T_{-x}$ be the first hitting time of $-x$ of $\X^{\lambda}(t)$, then
	$$
	\C(t) = \inf\left\{s: \int_0^s \frac{1}{x+\X^\lambda(u\wedge T_{-x}) } \,du = t\right\}.
	$$ This is a simple calculus exercise and the proof can be found in \cite[Section 2]{CPU13} and in \cite[Chapter VI, Section 1]{EK86}.
	
	We note that almost surely
	$$
	I:=\int_0^{T_{-x}} \frac{1}{x+\X^{\lambda}(u\wedge T_{-x})}\,du <\infty.
	$$ Indeed, this is true if we replace $\X^\lambda$ with a Brownian motion $B$ and Girsanov's theorem \cite[Chapter VIII]{RY99} implies that it is true almost surely for the $\X^\lambda$ as well. Hence $$
	\zeta = \inf\{t:x+\X^\lambda(\C(t)) = 0\} =\inf\{s: \C(s) = T_{-x}\} = I <\infty.
	$$
\end{proof}

\section{Convergence of the Cumulative Cousin Process} \label{sec:klim}

We begin by recalling the notation in Theorem \ref{thm:kconv}. The vertices are labeled by $w_n^k(j)$ for $j=0,1,\dotsm, m-1$ in a breadth-first order. Each vertex $w_n^k(j)$ at some height $h = \hgt_n^k(w_n^k(j))$ and for this $j$ and $h$, we have
$$
\csn_n^k(w_n^k(j)) = Z_n^k(h).
$$

\begin{proof}[Proof of Theorem \ref{thm:kconv} and Part (1) of Corollary \ref{cor:2}]
	Throughout the proof we let $k = k(n,x) = \fl{n^{1/3}x}$. We also omit reference to the vertex $w_n^k$ when using the $\csn$ statistic and write
	\begin{equation*}
	\csn_n^k (j) = \csn_n^k(w_n^k(j)). 
	\end{equation*} Recall that we have defined $\displaystyle C_n^k(h) = \sum_{j=0}^h Z_n^k(j)$. By Lemma \ref{thm:zconv} and the Skorohod representation theorem, we can and do assume that 
	\begin{equation*}
	\left(\left(n^{-1/3}Z_n^k(\fl{n^{1/3}t},n^{-2/3}C_n^k(\fl{n^{1/3}t}) \right);t\ge0 \right)\longrightarrow \left((\Z(t),\C(t));t\ge 0 \right),\qquad a.s.
	\end{equation*} where $(\Z,\C)$ is a weak solution of the stochastic differential equation in \eqref{eqn:zeqn.2}. By Lemma \ref{lem:compactsupport}, we can define
	\begin{equation*}
	\Z(\infty) = 0.
	\end{equation*}
	Let $A_n(k) := \sum_{h\ge 0} Z_{n}^k(h)$ be the total number of vertices ever infected.
	
	We assume that this occurs on the probability space $(\Omega,\F,\P)$. Due to the Skorohod representation theorem, Theorem 1 in \cite{MartinLof98}, and the equality in distribution between 
	\begin{equation*}
		\int_0^\infty \Z(s)\,ds = T_{-x} = \inf\{t: \X^\lambda(s) = -x\}
	\end{equation*} following from Lemma \ref{lem:uniquenessLemma}, without loss of generality we can assume that
	\begin{equation}
	\label{eqn:AnkAS}
	n^{-2/3}A_n(k)\to \int_0^\infty \Z(s)\,ds,\qquad \text{a.s.}
	\end{equation} 
	
	We break the proof down into several steps.
	\begin{enumerate}
		\item[Step 1:] Show that $\displaystyle\left(n^{-1/3}\csn_n^k \circ C_n^k (\fl{n^{1/3}r});r\ge0 \right) \to \left( \Z(r);r\ge 0 \right)$ a.s. in $\D(\R_+,\R_+)$.
		\item[Step 2:] Show that for large values of $t$, $\displaystyle  n^{-1/3}\csn_n^k(\fl{n^{2/3}t}) \to \Z(S_t(\C))$ a.s. in $\R$, for some time-change $S_t(\C)$. This convergence is in $\R$.
		\item[Step 3:] We argue the same convergence as step two holds for small values of $t$.
		\item[Step 4:] We argue that Steps 2 and 3 imply convergence in the Skorohod space.
		\item[Step 5:] We argue that $\displaystyle \int_0^{S_t(\C)} \Z(s)^2\,ds =  \int_0^{t\wedge T_{-x}} x+\X^\lambda (s)\,ds.$
	\end{enumerate}
	
	\textbf{Step 1:} With the convention that we start labeling the breadth-first order at $0$, it is a simple counting argument to see that $w_n^k(C_n^k(h))$ is the first vertex in the breadth-first ordering that is at distance $h+1$ from the root in its connected component. Thus $$
	\{ w_n^k(j): 0\le j \le C_n^k(h)-1\}  =\left\{v\in \G_n :\hgt_n^k(v)\le h\right\}.
	$$ See also Figure \ref{fig:comp}. Therefore, \begin{equation*}
\csn_n^k(C_n^k(r)) =Z_n^k(r).
	\end{equation*}  Consequently,
	\begin{equation*}
	\left(n^{-1/3}\csn_n^k(C_n^k(\fl{n^{1/3}r}));r\ge0\right)\longrightarrow \left(\Z(r);r\ge 0 \right)
	\end{equation*}

	\textbf{Step 2:} The next part of the proof mimics part of the proof of Theorem 1.5 in \cite{Duquesne09}. 
	We define for each $f\in \D(\R_+,\R_+)$ and $y \in \R_+$ the function
	$$
	S_y(f) = \inf\{t: f(t)\vee f(t-)>y\},
	$$ where $\inf\emptyset = \infty$. Set $\mathcal{V}(f) = \{y\in \R_+: S_{y-}(f) < S_{y}(f)\}$. By  Lemma 2.10 and Proposition 2.11 in  \cite[Chapter VI]{JS87}, for each fixed $y$, $f\mapsto S_y(f)$ is a measurable map from $\D(\R_+,\R_+)\to \R$ which is continuous at each $f$ such that $y\notin\mathcal{V}(f)$.

	Define $\zeta = \inf\{t:\Z(t) = 0\}$ which is finite by Lemma \ref{lem:compactsupport}. We observe that $\C(t)$ is a strictly increasing continuous function for all $t\in [0,\zeta)$, since its derivative is strictly positive, and $\C(t) = \C(\zeta) = \int_0^\infty \Z(s)\,ds$ for all $t\ge \zeta$. In particular, $S_{\C(\zeta)}(\C) = +\infty$ whereas $S_{\C(\zeta)-}(\C) = \zeta<\infty$. Therefore, we have $\mathcal{V}(\C) = \{\C(\zeta)\}$ a.s.
	
	We recall from \eqref{eqn:AnkAS} with the notation above, that for each $\eps>0$ there exists an $ N(\w)<\infty$ such that $n^{-2/3}A_n(k) < \C(\zeta)+\eps$ for all $n\ge N(\w)$. For those $n$ sufficiently large, we have almost surely and for each $t>\C(\zeta)+\eps$
	\begin{align*}
	n^{-1/3}\csn_n^k(\fl{n^{2/3} t}) &= 0=  \Z(S_t(\C)) 
	\end{align*} where we use the fact that $S_t(\C) = \infty$ for all $t\ge \C(\zeta) = \sup_{s} \C(s)$. Indeed, $A_n(k) < n^{2/3}(C(\zeta)+\eps)< n^{2/3}t$ and so $\csn_n^k(\fl{n^{2/3}t}) = 0$ by definition (see around equation \eqref{eqn:kDef}) for those $t$ sufficiently large.
	
	By taking $\eps\downarrow 0$, we have argued that almost surely 
	\begin{equation}\label{eqn:csnConv_bigt}
	n^{-1/3} \csn_n^k(\fl{n^{2/3}t}) \to \Z(S_t(\C))\qquad \forall t> \C(\zeta),
	\end{equation} where the convergence is convergence as real numbers. Let $\mathscr{N}_>\subset \Omega$ denote the null set for which the above state does not hold.
	
	\textbf{Step 3:} We now argue that \eqref{eqn:csnConv_bigt} also holds for $t<\C(\zeta)$. To begin, we define the process $V_n^k = (V_n^k(j);j=0,1\dotsm)$ by
	\begin{equation*}
	V_n^k(j) = \min\{h: C_n^k(h)> j\}.
	\end{equation*} We observe that if vertex $w_n^k (j)$ is at height $h$, then $V_n^k(j) = h$ because we start indexing the vertices at zero.  Hence,
	\begin{equation*}
	|C_n^k(V_n^k(j)) -j | \le Z_n^k(V_n^k(j)).
	\end{equation*} Indeed, the first labeled vertex at height $h$ is $w_n^k(C_n^k(h-1))$ and the last vertex of height $h$ is $w_n^k(C_n^k(h-1)+ Z_n^k(h) - 1)$. 
	In particular, 
	\begin{equation}\label{eqn:CV}
		\left|n^{-2/3} C_n^{k}(V^{k}_n(\fl{n^{2/3}t})) - t\right| \le \sup_{h\ge 0} n^{-2/3} Z_n^k(h) \longrightarrow  \to 0.
	\end{equation}
	
	We now look at the events
	$$
	E_q = \{q<\C(\zeta)\}.
	$$ We observe that for there exists a null set $\mathscr{N}_q\subset E_q$, such that for each $\w\in E_q\setminus\mathscr{N}_q$ and for every $t\in[0,q]$ we have the following convergence of real numbers 
	\begin{equation}\label{eqn:vconv}
	n^{-1/3}V_n^k(\fl{n^{2/3} t}) = S_t \left( n^{-2/3}C_n^k(\fl{n^{1/3}\cdot})\right) \longrightarrow S_t(\C),
	\end{equation} by the aforementioned continuity of $S_t(\cdot)$ at those $f\in\D(\R_+,\R_+)$ with $t\notin\mathcal{V}(f)$. We therefore have for each $\w\in E_q\setminus \mathscr{N}_q$ 
	$$
	n^{-1/3}\csn_n^k\left(C_n^k\left( V_n^k(\fl{n^{2/3}t})\right) \right) \to \Z(S_t(\C)),\qquad \forall t\in[0,q]
	$$ Indeed, this follows from \cite[Lemma pg 151]{Billingsley99} and the observation $n^{-1/3}\csn_n^k \circ C_n^k(\fl{n^{1/3}\cdot})$ converges in the $J_1$ topology to the continuous function $t\mapsto \Z(S_t(\C))$ for $t\in[0,\zeta)$.
	Still working with an $\w\in E_q\setminus \mathscr{N}_q$, we now observe that
	\begin{align*}
	&\left| n^{-1/3} \csn_n^k \left(C_n^k\left( V_n^k(\fl{n^{2/3}t})\right) \right)  - n^{-1/3} \csn_n^k(\fl{n^{2/3} t}) \right| \longrightarrow 0  ,
	\end{align*} by equation \eqref{eqn:CV} and \cite[Lemma pg. 151]{Billingsley99}. By taking the union over all $q\in \mathbb{Q}\cap\R_+$, we have argued Step 3.
	
	\textbf{Step 4:} We have shown that outside of the null set $\mathscr{N} = \mathscr{N}_> \cup \bigcup_{q\in \mathbb{Q}\cap\R_+} \mathscr{N}_q$ that
	\begin{equation}\label{eqn:csnconv_allt}
	n^{-1/3}\csn_n^k (\fl{n^{2/3}t})  \longrightarrow \Z(S_t(\C)) ,\qquad \forall t\neq \zeta(\w),
	\end{equation} where the convergence is a real numbers. 
	
	We now argue that for each such $\w\in \Omega\setminus \mathscr{N}$ and each $t\neq \zeta(\w)$
	\begin{equation}\label{eqn:csnjumps}
	\sum_{0\le s\le t}\left( n^{-1/3}\csn_n^k(\fl{n^{2/3}s}) - n^{-1/3}\csn_n^k(\fl{n^{2/3}s-}) \right)^2 \longrightarrow 0.
	\end{equation}
	Towards this end we have the following string of inequalities
	\begin{align*}
	\sum_{0\le s\le t}&\left( n^{-1/3}\csn_n^k(\fl{n^{2/3}s}) - n^{-1/3}\csn_n^k(\fl{n^{2/3}s-}) \right)^2 \le n^{-2/3} \sum_{h = 0}^{V_n^k(\fl{n^{2/3}t})} \left(Z_n^k(h)-Z_n^k(h-1)\right)^2\\
	&= \int_0^{V_n^k(\fl{n^{2/3}t})+1} \left( Z_{n}^k (\fl{u}) - Z_n^k(\fl{u}-1\right)^2\,du\\
	&= \int_0^{n^{-1/3} V_n^k (\fl{n^{2/3} t}) + n^{-1/3} } \left(n^{-1/3}Z_n^k(\fl{n^{1/3}u}) - n^{1/3}Z_n^k(\fl{n^{1/3}u}-1) \right)^2\,du \\
	&\longrightarrow \int_0^{S_t(\C)} \left(\Z(u)-\Z(u)\right)^2\,du = 0.
	\end{align*} The first inequality comes from examining the jumps of $\csn_n^k(\fl{n^{2/3}s})$ occur when $h = n^{2/3}s \in \mathbb{Z}$ and the jump is of size $Z_n^{k}(h)-Z_n^{k}(h-1)$. The a.s. convergence for the integral follows from the time-change lemma in \cite[pg. 151]{Billingsley99} and the convergence in \eqref{eqn:vconv}. 
	
	By Theorem 2.15 in \cite[Chapter VI]{JS87}, equations \eqref{eqn:csnconv_allt} and \eqref{eqn:csnjumps}  imply that for all $\w\in \Omega\setminus \mathscr{N}$
	\begin{equation*}
	\left(n^{-1/3} \csn_n^k(\fl{n^{2/3}t}); t\ge0 \right) \longrightarrow \left( \Z(S_t(\C)) ; t\ge0\right) \qquad \text{ in } \D(\R_+,\R).
	\end{equation*}
	
	\textbf{Step 5:} We let $\X^\lambda$ be the Brownian motion with parabolic drift related to the processes $(\Z,\C)$ in Lemma \ref{lem:uniquenessLemma}(2). The theorem follows from the following change of variables
	\begin{align*}
	\Z(S_t(C)) &=   x+ \X^\lambda (\C(S_t(\C))) = x+ \X({t\wedge T_{-x}})
	\end{align*} where we used the relationship $\C(S_t(\C)) = t\wedge C(\zeta) = t \wedge T_{-x}$.
	
	The rest of Theorem \ref{thm:kconv} now follows from integration.
\end{proof}

\section{A Self-Similarity Result}\label{sec:ss}

We first observe the following relationship in $\lambda$ for the process $\X^\lambda$. Namely,
\begin{equation*}
\left(\left( \X^\lambda(t_0+t) - \X^\lambda(t_0); t\ge 0  \right)\bigg| \X^\lambda(t_0) = \inf_{s\le t_0} \X^\lambda(s) \right) \overset{d}{=} \left( \X^{\lambda-t_0}(t);t\ge0 \right).
\end{equation*} This observation was used by Aldous \cite{A97_1} to simplify the description of the (time-inhomogeneous) excursion measure of $\X^\lambda$ at time $t$ to the excursion measure of $\X^{\lambda-t}$ at time 0. See also \cite{ABG12}.

A similar result will hold in our situation as well. We state it in the following theorem
\begin{thm}
	Let $\Z^{\lambda}_x(t), \C^{\lambda}_x(t)$ denote the solution to 
	\begin{equation*}\begin{split}
	d\Z^{\lambda}_x (t) &= \sqrt{\Z^\lambda_x(t)}\,dW(t)+ \left(\lambda - \C^{\lambda}_x(t) \right) \Z^{\lambda}_x(t)\,dt,\qquad \Z^\lambda_x(0) = x\\
	d\C^\lambda_x(t) &= \Z^\lambda_x(t)\,dt \qquad \C^{\lambda}_x(0) = 0
	\end{split}.
	\end{equation*}
	
	Then the following self-similarity result holds for any $t_0>0$, $z>0$ and $\mu>0$:
	\begin{equation}\label{eqn:selfsim.1}
	\left( \left(\left(\Z^{\lambda}_x(t_0+t), \C^{\lambda}_x(t_0+t)\right);t\ge0 \right) \Big| \Z^\lambda_x(t_0) = z, \C^{\lambda}_x(t_0) = \mu\right) \overset{d}{=} \left(\left(\Z^{\lambda-\mu}_z(t),\C^{\lambda-\mu}_z(t)\right);t\ge 0\right)
	\end{equation}
\end{thm}

\begin{proof}
	The proof follows from the decomposition in Lemma \ref{lem:uniquenessLemma}, particularly in equation \eqref{eqn:ctsTC}. Namely, there exists a Brownian motion with parabolic drift $\X^\lambda(t)$  such that
	\begin{equation*}
	\Z^\lambda_x(t) = x+ \X^{\lambda}(\C^\lambda_x(t) \wedge T_{-x}).
	\end{equation*}
	We also observe that
	\begin{align*}
	\X^\lambda(s_0+s)  &= B(s_0+s)+\lambda(s_0+s) -\frac{1}{2}(s_0+s)^2\\
	&= B(s_0)+B(s_0+s)-B(s_0) + \lambda s_0 + \lambda s- \frac{1}{2}s_0^2 - s_0 s - \frac{1}{2}s^2\\
	&= \X^\lambda(s_0) + \left(B(s_0+s)-B(s_0) + (\lambda-s_0) s - \frac{1}{2}s^2 \right)\\
	&= \X^\lambda(s_0) + \tilde{\X}^{\lambda-s_0}(s),
	\end{align*} for a process $\tilde{\X}^{\lambda-s_0} \overset{d}{=} \X^{\lambda-s_0}$ which is independent of $\sigma\left\{\X^\lambda(u); u\le s_0 \right\}$.
	Hence, we have
	\begin{align*}
	\Z^\lambda_x(t_0+t)&= x + \X^\lambda\left(\C^\lambda_x(t_0+t)\right)\\
	&= x+ \X^\lambda\left(\C^\lambda_x(t_0) + \int_0^t \Z_x^\lambda(t_0+s)\,ds \right) \\
	&= x + \X^\lambda(\C_x^\lambda(t_0)) +  \tilde{B}\left(\int_0^t \Z^\lambda_x(t_0+s)\,ds\right)  \\
	&\qquad\qquad\qquad\qquad+ \left(\lambda-\C_x^\lambda(t_0)\right) \int_0^t \Z^\lambda_x(t_0+s)\,ds -  \frac{1}{2}\left(\int_0^t \Z^\lambda_x(t_0+s)\,ds\right)^2,
	\end{align*} where $\tilde{B}$ is a Brownian motion independent of $\sigma\{\X^\lambda(u): u\le \C(t_0)\}$. Hence, conditionally on $\Z^\lambda_x (t_0)  = z$ and $\C^\lambda_x(t_0) = \mu$ gives
	$$
	\Z^\lambda_x(t_0+t) = z + \tilde{\X}^{\lambda- \mu}\left(\int_0^t \Z^\lambda_x(t_0+s)\,ds \right)
	$$ By Lemma \ref{lem:uniquenessLemma}, this is equivalent to the statement in \eqref{eqn:selfsim.1} .
\end{proof}

\section{A More General Asymptotic Regime}\label{sec:theta}

As observed by Bollob\'{a}s in \cite{B85}, the asymptotic order of largest component of the Erd\H{o}s-R\'{e}nyi random graph $G(n,n^{-1}+\lambda\log(n)^{1/2}n^{-4/3})$ is $n^{2/3}(\log n)^{1/2}$ as $n\to\infty$. Actually, he proves a much more general result, but we will not state that fully here. We instead examine a more general asymptotic regime.

We consider any sequence of real numbers $\theta_n$ such that 
\begin{equation}
\label{eqn:thetan}
\theta_n = o(n^{1/3}), \qquad\text{and}\qquad \lim_{n\to\infty} \theta_n =\infty.
\end{equation} In the introduction we used the notation $\eps_n$ instead of $\theta_n$. The conditions in \eqref{eqn:thetan} can be reformulated for $\eps_n$ in the statement of Theorem \ref{thm:kconv_gen} by setting $$
\theta_n = n^{1/3} \eps_n.
$$ We also fix a $\lambda\in \R$ and let \begin{equation*}
\G_n^\theta = G(n,n^{-1}+\lambda \theta_n n^{-4/3}).\end{equation*}

To distinguish the notation, we let $Z^{\theta,k}_n(h)$ denote the height profile of $\G_n^\theta$ starting from $k$ uniformly chosen vertices (see Section \ref{sec:labeling} for more information on how this is constructed). With this notation, we can state the following theorem:
\begin{lem} \label{thm:genz} Fix $x>0$.
	Suppose that $\theta_n$ satisfies \eqref{eqn:thetan} and $k = k(n) = \fl{ \theta_n^2 n^{1/3}x}$. Then the following convergence holds in the Skorohod space $\D(\R_+,\R_+)$
	\begin{equation}\label{eqn:ztheta}
	\left( \frac{1}{\theta_n^2 n^{1/3}} Z_n^{\theta,k} \left( \fl{ \theta_n^{-1} n^{1/3} t}\right) ;t\ge0\right) \Longrightarrow \left(z(t);t\ge0 \right),
	\end{equation} where $z$ solves the deterministic equation
	\begin{equation}\label{eqn:detz}
	z(t) = f\left(\int_0^t z(s)\,ds\right),\qquad f(t) =x + \lambda t -\frac{1}{2}t^2.
	\end{equation}
\end{lem}

The proof follows from lemmas similar to the lemmas found in Section \ref{sec:zthm}. Before stating those lemmas, we make some comments on the solution $z(t)$ found in \eqref{eqn:detz}. We have already mentioned that $$
c(t) = \int_0^t z(s)\,ds = \inf\{s: \int_0^s \frac{1}{f(u)}\,du = t\}.
$$ See also, \cite[Section 2]{CPU13} and \cite[Section 6.1]{EK86} for more details on time changes. We have ``$=t$'' instead of ``$>t$" because the inverse is actually a two-sided inverse. Indeed, since $\int_0^{t_0} \frac{1}{f(u)}\,du = \infty$ where $t_0 = \lambda+\sqrt{2x+\lambda^2}$ is the largest root of $f(t)$, the function $c$ is strictly increasing continuous function $c:[0,\infty)\to [0,\lambda+\sqrt{2x+\lambda^2})$. The function $c$ can actually be explicitly computed:
\begin{equation*}
c(t) = \lambda + \sqrt{2x+\lambda^2} \tanh\left(\frac{\sqrt{2x+\lambda^2}}{2} t + \operatorname{arctanh}\left(\frac{-\lambda}{\sqrt{2x+\lambda^2}} \right) \right).
\end{equation*}

We also make comments on the scaling found in Theorem \ref{thm:genz}. In order to describe this scaling, we introduce the diameter of the graph $\G_n^\theta$, as
\begin{equation*}
\mathscr{D}_n^\theta = \mathscr{D}_n^{\theta_n} = \max_{u,v\in \G_n^\theta}\left\{ \dist(u,v): \dist(u,v)<\infty\right\}.
\end{equation*} The trivial observation is that $Z_n^{\theta,k}(h)>0$ implies that $\mathscr{D}^\theta_n\ge h$. 
A result of {\L}uczak \cite[Theorem 11(iii)]{Luczak98} implies when $\lambda<0$ that
\begin{equation*}
\mathscr{D}_n^\theta = \frac{\log(2\theta_n^3)+O(1)}{-\log(1-{\theta_n}n^{-1/3})} 
\end{equation*} with high probability as $n\to\infty$. There is a typo in the statement of Theorem \cite[Theorem 11(iii)]{Luczak98}, he writes an $\log(2\eps^2n)$ term when there should be an $\log(2\eps^3n)$ term. In the supercritical ($\lambda>0$) regime, it appears that the work of Ding, Kim, Lubetzky and Peres \cite{DKLP10,DKLP11} provide more precise results. Namely, they show \cite[Theorem 1.1]{DKLP11} that if $\c_n^\theta$ is the largest component of $\G_n^\theta$, for $\lambda>0$, then with high probability
\begin{equation*}
\text{diam}(\c_n^\theta) = (3+o(1))n^{1/3}\theta_n^{-1} \log (\theta_n^3)\qquad\text{as }n\to\infty.
\end{equation*}
Even more precise asymptotic result in this regime can be found in \cite{RW10}, again in the supercritical regime when $\lambda>0$. 

Both of these results on the asymptotic diameter $\mathscr{D}_n^\theta$ suggest the proper ``time'' scaling in Theorem \ref{thm:genz} should be $\theta_n^{-1}n^{1/3}\log(\theta_n)t$ as compared with $\theta_n^{-1}n^{1/3}t$; however, this is not the correct scaling to obtain a non-trivial limit.

\subsection{Lemmas}

In the connection to the Reed-Frost model of epidemics, it is easy to see that the analog of \eqref{eqn:condDist} becomes the following
\begin{equation*}
\left(Z_n^{\theta,k}(h+1) \big| Z_n^{\theta,k}(h) = z, C_n^{\theta,k}(h) = c \right) \overset{d}{=} \left\{  \begin{array}{ll}\text{Bin} \left(n-c, q_\theta(n,z) \right) &: z>0, c<n\\
0 &: \text{else} 
\end{array}\right.,
\end{equation*} where $q_\theta(n,z)$ is defined as
\begin{equation*}
q_\theta(n,z) = 1-\left(1- n^{-1} -\lambda \theta_n n^{-4/3} \right)^z
\end{equation*}
The analog of Lemma \ref{lem:asympt} becomes the following
\begin{lem}\label{lem:asympt2}
	Let $\beta_\theta(n,z,c)$ denote a $\text{Bin}(n-c,q_\theta(n,z))$ random variable. Let $\mu_\theta,\sigma_\theta^2,\kappa_\theta$ denote the statistics in \eqref{eqn:betaStats} with $\beta_\theta$ replacing $\beta$. Fix $r>0$ and $T>0$ and define $$
	\Omega_n^\theta = \Omega_n^\theta(n,r,T):= \left\{(z,c)\in \mathbb{Z}^2: 0\le z\le n^{1/3}\theta_n^2 r, 0\le c\le n^{2/3}\theta_n rT \right\}
	$$ then the following bounds hold
	\begin{equation*}
	\begin{split}
	&\sup_{\Omega_n^\theta} \left|\mu_\theta(n,z,c) - z - n^{-1/3}z(\lambda \theta_n - n^{-2/3}c) \right| = O\left( \theta_n^4 n^{-1/3}+1\right)\\
	&\sup_{\Omega_n^\theta} \left|\sigma_\theta^2(n,z,c) - z - n^{-1/3}z(\lambda \theta_n - n^{-2/3}c) \right| = O\left( \theta_n^4 n^{-1/3}+1\right)\\
	&\sup_{\Omega_n^\theta} |\kappa_\theta(n,z,c) | = O(\theta_n^{12}+ \theta_n^8 n^{1/3} +\theta_n^4 n^{2/3}) 
	\end{split}
	\end{equation*}
\end{lem}

\begin{proof}
	The proofs of the convergence of $\mu_\theta$ and $\sigma_\theta^2$ follow from the same argument as in the proof of Lemma \ref{lem:asympt}, and we omit it here.
	
	We do argue the result for $\kappa_\theta$ since it is much more involved computationally. We again use the expansion:
	\begin{align*}
	\kappa_\theta(n,z,c) &= \E\left[(\beta_\theta(n,z,c)-\mu_\theta(n,z,c))^4 \right] + 4 \E\left[(\beta_\theta(n,z,c)-\mu_\theta(n,z,c))^3\right](\mu_\theta(n,z,c)-z) \\
	&\qquad + 6 \E\left[ (\beta_\theta(n,z,c)-\mu_\theta(n,z,c))^2 \right](\mu_\theta(n,z,c)-z)^2  \\
	&\qquad + 4\E\left[\beta_\theta(n,z,c)-\mu_\theta(n,z,c)\right](\mu_\theta(n,z,c)-z)^3\\
	&\qquad +(\mu_\theta(n,z,c)-z)^4\\
	&=: \kappa_{4,\theta}(n,z,c) + 4\kappa_{3,\theta}(n,z,c) + 6 \kappa_{2,\theta}(n,z,c) +0 + \kappa_{0,\theta}(n,z,c).
	\end{align*}
	
	We can use the bound for $\mu_\theta$ and Minkowski's inequality to get
	\begin{align*}
	\sup_{\Omega_n^\theta}\left|\kappa_{0,\theta}(n,z,c) \right| &\left(\mu_\theta(n,z,c) - z \right)^4\\
	&\le \left[\sup_{\Omega_n^\theta} \left|n^{-1/3} z(\lambda\theta_n-n^{-2/3}c) \right| + O(\theta_n^4n^{-1/3} + \theta_n^2 n^{-2/3})\right]^4\\
	&\le C \left( \sup_{\Omega_n^\theta} |n^{-1/3} z(\lambda\theta_n - n^{-2/3}c)|^4 + O(\theta_n^{16} n^{-4/3} + \theta_n^8 n^{-8/3}) \right)\\
	&= O\left(\theta_n^{12}  + \theta_n^{16} n^{-4/3} + \theta_n^8 n^{-8/3} \right) \le O(\theta_n^{12})
	\end{align*} where in the last inequality we used the bounds in \eqref{eqn:thetan}.
	
	The next three follow from the bounds below. They are easy to verify using the original bounds of $\sigma_\theta^2$ and $\mu_\theta$, and computations similar to the one above:
	\begin{align*}
	&\sup_{\Omega_n^\theta} \left|\mu_\theta(n,z,c) - z \right| = O\left(\theta_n^3 \right)\\
	&\sup_{\Omega_n^\theta} \left|\sigma_\theta^2(n,z,c) \right| = O\left(\theta_n^2 n^{1/3} + \theta_n^3 + \theta_n^4 n^{-1/3} \right)\\
	&\qquad\qquad\qquad\qquad  =  O\left(\theta_n^2 n^{1/3} \right)
	\end{align*}
	
	Using the same expansions as in Lemma \ref{lem:asympt}, we have
	\begin{align*}
	\sup_{\Omega_n^\theta} |\kappa_{2,\theta}(n,z,c)| &= O(\theta_n^6)\times O(\theta_n^2 n^{1/3}) = O(\theta_n^8 n^{1/3})\\
	\sup_{\Omega_n^\theta} |\kappa_{3,\theta}(n,z,c)| &= O(\theta_n^2 n^{1/3}) \times O(\theta_n^3) = O(\theta_n^8 n^{1/3})\\
	\sup_{\Omega_n^\theta} |\kappa_{4,\theta}(n,z,c)|& = O(\theta_n^2 n^{1/3})^2 = O(\theta_n^4 n^{2/3})
	\end{align*}
	
	This proves the desired bounds.
\end{proof}

One can use the bounds in the lemma above to prove an analog of Lemma \ref{lem:ekVerify}. We first establish some notation.
We now let $\F_n^{\theta,k}(h) = \sigma(Z_n^{\theta,k}(j),j\le h)$ be the filtration generated by $Z_n^{\theta,k}$ and let $Z_n^{\theta,k}(h) = M_n^{\theta,k}(h) + B_n^{\theta,k}(h)$ be the decomposition of $Z_n^{\theta,k}$ into an $\F_n^{\theta,k}(h)$-martingale $M_n^{\theta,k}$ and a process $B_n^{\theta,k}$. We also let $Q_n^{\theta,k}$ be the process which makes $(M_n^{\theta,k}(h))^2 - Q_n^{\theta,k}(h)$ an $\F_n^{\theta,k}(h)$-martingale. Define the rescaled processes, in comparison to \eqref{eqn:rescales},
\begin{equation}\label{eqn:rescales2} 
\begin{split}
\tilde{Z}_n^{\theta,k}(t) &= \theta_n^{-2} n^{-1/3}Z_n^{\theta,k}(\fl{ \theta_n^{-1} n^{1/3} t})\quad \tilde{C}_n^{\theta,k}(t) = \theta_n^{-1}n^{-2/3}C_n^{\theta,k}(\fl{ \theta_n^{-1} n^{1/3} t})\qquad \\
\tilde{M}_n^{\theta,k}(t) &= \theta_n^{-2}n^{-1/3}M^{\theta,k}_n(\fl{ \theta_n^{-1} n^{1/3} t})\quad \tilde{B}_n^{\theta,k}(t) =  \theta_n^{-2}n^{-1/3}B^{\theta,k}_n(\fl{ \theta_n^{-1} n^{1/3} t}) \\
\tilde{Q}_n^{\theta,k}(t)&= \theta_n^{-4} n^{-2/3}Q^{\theta,k}_n(\fl{ \theta_n^{-1} n^{1/3} t}).
\end{split}.
\end{equation} Also define $\tau_n^{\theta,k}(r) = \inf\{t:\tilde{Z}_n^{\theta,k}(t)\vee \tilde{Z}_n^{\theta,k}(t-)>r\}$ and $\hat{\tau}_n^{\theta,k}(r) = \theta_n n^{-1/3}\inf\{k: Z_n^{\theta,k}(h)>\theta_n^2 n^{1/3}r\}$.

The analog of Lemma \ref{lem:ekVerify} is the following lemma. The proof is omitted since it is similar to the proof of Lemma \ref{lem:ekVerify}.
\begin{lem} \label{lem:ekVerify2} Fix any $r>0$, $T>0$ and $x>0$. Let $k = k(n) = \fl{\theta_n^2 n^{1/3}x}$. 
	The following limits hold
	\begin{enumerate}
		\item $\displaystyle \lim_{n\to\infty}\E\left[ \sup_{t\le T\wedge \tau_n^{\theta,k}(r)} |\tilde{Z}_n^{\theta,k}(t)-\tilde{Z}_n^{\theta,k}(t-)|^2\right]  = 0$.
		\item $\displaystyle \lim_{n\to\infty}\E\left[ \sup_{t\le T\wedge \tau_n^{\theta,k}(r)} |\tilde{B}_n^{\theta,k}(t)-\tilde{B}_n^{\theta,k}(t-)|^2\right]  = 0$.
		\item $\displaystyle \lim_{n\to\infty}\E\left[ \sup_{t\le T\wedge \tau_n^{\theta,k}(r)} |\tilde{Q}_n^{\theta,k}(t)-\tilde{Q}_n^{\theta,k}(t-)| \right] = 0$.
		\item $\displaystyle \sup_{t\le T\wedge \tau_n^{\theta,k}(r)} \left|\tilde{Q}_n^{\theta,k}(t)  \right| \longrightarrow 0$, as $n\to\infty$ almost surely
		\item $\displaystyle  \sup_{t\le T\wedge \tau_n^{\theta,k}(r)}\left| \tilde{B}_n^{\theta,k}(t)- \int_0^t (\lambda -\tilde{C}_n^{\theta,k}(s))\tilde{Z}_n^{\theta,k}(s) \,ds\right|\overset{P}{\longrightarrow}0$, as $n\to\infty$.
	\end{enumerate}
\end{lem}

Finally, using the machinery of \cite[Chapter 7]{EK86}, in particular Theorem 7.4.1, Lemma \ref{thm:genz} follows from Lemma \ref{lem:ekVerify2}.

\subsection{The Cumulative Cousin Process and Corollaries}

Just as we examined cousin process and the cumulative cousin process of the Erd\H{o}s-R\'{e}nyi random graph $\G_n$ and obtained a non-trivial rescaled limit, we get a similar result in this newer regime.

\begin{proof}[Proof of Theorem \ref{thm:kconv_gen} and Part (2) of Corollary \ref{cor:2}] We write $\theta_n = \eps_n n^{1/3}$. We write $C_n^{\theta,k}(h) = \sum_{j\le h} Z_n^{\theta,k}(j)$. The proof of the scaling of the $\csn$ statistic follows from a similar argument as in the proof of Theorem \ref{thm:kconv} with the only replacements being the scaling, and the scaling limits. We omit that here, but include the proof of the cumulative cousin statistics which does not rely on just integration.

Just as in the proof of Theorem \ref{thm:kconv}, we can write
\begin{align*}
K_n^{\eps,k}\circ C_n^{\theta,k}(h) = \sum_{\ell=0}^h \left(Z_n^{\theta,k}(\ell)\right)^2.
\end{align*}

Then 
\begin{align*}
\frac{1}{\eps_n^3 n^{2}}K_n^{\eps,k} \circ C_n^{\theta,k}(\fl{\eps_n^{-1}t}) &=\frac{1}{\theta_n^3 n} \int_0^{\fl{n^{1/3}\theta_n^{-1}t}} \left(Z_n^{\theta,k}(\fl{u}) \right)^2\,du\\
&=\frac{1}{\theta_n^3 n}\int_0^t \left(Z_n^{\theta,k}(\fl{n^{1/3}\theta_n^{-1} s}) \right)^2  n^{1/3} \theta_n^{-1}\,ds + o(1)\\
&=\int_0^t \frac{1}{\theta_n^4 n^{2/3}}\left( Z_n^{\theta,k}(\fl{n^{1/3} \theta_n^{-1}s}) \right)^2\,ds + o(1)
\\
& \Longrightarrow \int_0^t \left(z(s)\right)^2\,ds,
\end{align*} where the $o(1)$ term vanishes.

Just as in Step 3 of the proof of Theorem \ref{thm:kconv}, we can go from the convergence above to the convergence
\begin{equation*}
\left(\frac{1}{\eps_n^3 n^{8/3}} K_n^{\eps,k} (\fl{\eps_n n t});t\ge 0\right) \Longrightarrow \left(\int_0^{\inf\{u: c(u)>t\}} z(s)^2\,ds\right),
\end{equation*} where $c(t) = \int_0^t z(s)\,ds$. However, by \eqref{eqn:detz} and the paragraph thereafter,$$
z(t) = f\circ c(t), \qquad \text{where }f(t) = x+\lambda t -\frac{1}{2}t^2,
$$ and $c:[0,\infty) \to [0,\lambda + \sqrt{2x+\lambda^2})$.
Hence, by \cite[Chapter 0]{RY99},
$$
\int_0^{\inf\{u:c(u)>t\}} z(s)^2\,ds = \int_0^{c^{-1}(t)} f(c(s))\,dc(s) = \int_0^{t\wedge \lambda + \sqrt{2x+\lambda^2}} f(s)\,ds = xt+\frac{\lambda t^2}{2}-\frac{t^3}{6} \vee 0.
$$
\end{proof}

We can prove Proposition \ref{prop:components}, which is just a corollary of Lemma \ref{thm:genz}.

\begin{proof}[Proof of Proposition \ref{prop:components}]
The proof comes from the following general observation. If $f_n,f\in \D(\R_+,\R_+)\cap L^1(\R_+,dx)$ and $f_n\to f$ in the $J_1$ topology then
$$
\lim_{n\to\infty} \int_0^\infty f_n(t)\,dt \ge \lim_{n\to\infty} \int_0^T f_n(t)\,dt = \int_0^T f(t)\,dt.
$$ By taking $T$ large enough, once can make $\int_0^T f\,dt$ arbitrarily close to $\int_0^\infty f(s)\,ds.$

The proof is finished by the following observation, where we again write $\theta = \theta_n = n^{1/3}\eps_n$
\begin{align*}
n^{-1/3} \eps_n A_n^\eps(k) &= n^{-1/3}\eps_n\sum_{h\ge 0} Z_n^{\theta,k}(h)\\
&= n^{-2/3}\theta_n \sum_{h\ge 0} Z_n^{\theta,k}(h)\\
&= n^{-2/3} \theta_n \int_0^\infty Z_n^{\theta,k}(\fl{u})\,du\\
&= n^{-2/3} \theta_n \int_0^\infty Z_n^{\theta,k}(\fl{\theta_n^{-1}n^{1/3}}) \frac{n^{1/3}}{\theta_n}\,dt\\
&= \int_0^\infty \tilde{Z}_n^{\theta,k}(t)\,dt.
\end{align*} 
\end{proof}

\subsection{A Conjecture}

The scaling found in Corollary 1 in \cite{CPU17} tells us that under reasonable conditions, see \cite{CPU13}, if a breadth first walk $X_n = (X_n(k);k=0,1,\dotsm)$ has a rescaled limit in the Skorohod space
\begin{equation*}
\left(\frac{\alpha_n}{\gamma_n} X_n(\fl{\gamma_n t});t\ge0 \right) \Longrightarrow \left( X(t);t\ge0 \right)
\end{equation*} then the process $Z_n = (Z_n(h);h\ge0)$ defined as a solution to the difference equation
\begin{equation}\label{eqn:lamp.disc}
Z_n(h) = X_n\circ C_n(h-1),\qquad C_n(h) = \sum_{j=0}^h Z_n(j),
\end{equation} has the rescaled limit
\begin{equation}\label{eqn:lamp.cts}
\left(\frac{\alpha_n}{\gamma_n} Z_{n}(\fl{\alpha_n t});t\ge 0 \right) \Longrightarrow \left(Z(t);t\ge0 \right)
\end{equation} where is the unique solution to $$
Z(t) = X\left(\int_0^t Z(s)\,ds\right).
$$

Even though we have no breadth-first walk in this work where we can apply the discrete Lamperti transform \eqref{eqn:lamp.disc}, we did get the continuous analog \eqref{eqn:lamp.cts} and the breadth-first walk in \cite{A97_1} to formulate Lemma \ref{thm:zconv}. This is precisely the content of parts (2) and (3) of Lemma \ref{lem:uniquenessLemma}. We can ask the question, does the breadth-first walk for $\G_n^\theta$ which is constructed as Aldous constructs his walk in \cite{A97_1} satisfy a scaling limit? We formulate this as a conjecture:
\begin{conj}
	Suppose that $\theta_n$ satisfies \eqref{eqn:thetan}. Let $X_n = (X_n(k);k = 0,1,\dotsm)$ be the breadth-first walk on $\G_n^\theta$ described in \cite{A97_1} for the $\G_n$ model. Then, in the Skorohod space $\D(\R_+,\R)$ the following convergence holds
	\begin{equation*}
	\left(\frac{1}{n^{1/3} \theta_n^2} X_n(\fl{n^{2/3} \theta_n t});t\ge0 \right) \Longrightarrow \left(\lambda t -\frac{1}{2}t^2 ;t\ge0 \right)
	\end{equation*}
\end{conj}

\section*{Acknowledgement}

The author would like to thank David Aldous for suggesting the problem that lead to the paper and Soumik Pal for continued guidance during this project. The author would also like to thank Louigi Addario-Berry for suggesting to prove the convergence of $\csn$ statistic instead of its cumulative sum.

\bibliographystyle{abbrv}

\begin{thebibliography}{10}
	
	\bibitem{ABG10}
	L.~Addario-Berry, N.~Broutin, and C.~Goldschmidt.
	\newblock Critical random graphs: limiting constructions and distributional
	properties.
	\newblock {\em Electron. J. Probab.}, 15:no. 25, 741--775, 2010.
	
	\bibitem{ABG12}
	L.~Addario-Berry, N.~Broutin, and C.~Goldschmidt.
	\newblock The continuum limit of critical random graphs.
	\newblock {\em Probab. Theory Related Fields}, 152(3-4):367--406, 2012.
	
	\bibitem{A91}
	D.~Aldous.
	\newblock The continuum random tree. {I}.
	\newblock {\em Ann. Probab.}, 19(1):1--28, 1991.
	
	\bibitem{A90}
	D.~Aldous.
	\newblock The continuum random tree. {II}. {A}n overview.
	\newblock In {\em Stochastic analysis ({D}urham, 1990)}, volume 167 of {\em
		London Math. Soc. Lecture Note Ser.}, pages 23--70. Cambridge Univ. Press,
	Cambridge, 1991.
	
	\bibitem{A93}
	D.~Aldous.
	\newblock The continuum random tree. {III}.
	\newblock {\em Ann. Probab.}, 21(1):248--289, 1993.
	
	\bibitem{A97_1}
	D.~Aldous.
	\newblock Brownian excursions, critical random graphs and the multiplicative
	coalescent.
	\newblock {\em Ann. Probab.}, 25(2):812--854, 1997.
	
	\bibitem{BM90}
	A.~Barbour and D.~Mollison.
	\newblock Epidemics and random graphs.
	\newblock {\em Lect. Notes Biomath.}, 86:86--89, 01 1990.
	
	\bibitem{BSW17}
	S.~Bhamidi, S.~Sen, and X.~Wang.
	\newblock Continuum limit of critical inhomogeneous random graphs.
	\newblock {\em Probab. Theory Related Fields}, 169(1-2):565--641, 2017.
	
	\bibitem{BHS18}
	S.~Bhamidi, R.~van~der Hofstad, and S.~Sen.
	\newblock The multiplicative coalescent, inhomogeneous continuum random trees,
	and new universality classes for critical random graphs.
	\newblock {\em Probab. Theory Related Fields}, 170(1-2):387--474, 2018.
	
	\bibitem{Billingsley99}
	P.~Billingsley.
	\newblock {\em Convergence of probability measures}.
	\newblock Wiley Series in Probability and Statistics: Probability and
	Statistics. John Wiley \& Sons, Inc., New York, second edition, 1999.
	\newblock A Wiley-Interscience Publication.
	
	\bibitem{B84a}
	B.~Bollob\'{a}s.
	\newblock The evolution of random graphs.
	\newblock {\em Trans. Amer. Math. Soc.}, 286(1):257--274, 1984.
	
	\bibitem{B84b}
	B.~Bollob\'{a}s.
	\newblock The evolution of sparse graphs.
	\newblock In {\em Graph theory and combinatorics ({C}ambridge, 1983)}, pages
	35--57. Academic Press, London, 1984.
	
	\bibitem{B85}
	B.~Bollob\'{a}s.
	\newblock {\em Random graphs}.
	\newblock Academic Press, Inc. [Harcourt Brace Jovanovich, Publishers], London,
	1985.
	
	\bibitem{CLU09}
	M.~E. Caballero, A.~Lambert, and G.~Uribe~Bravo.
	\newblock Proof(s) of the {L}amperti representation of continuous-state
	branching processes.
	\newblock {\em Probab. Surv.}, 6:62--89, 2009.
	
	\bibitem{CPU13}
	M.~E. Caballero, J.~L. P\'{e}rez~Garmendia, and G.~Uribe~Bravo.
	\newblock A {L}amperti-type representation of continuous-state branching
	processes with immigration.
	\newblock {\em Ann. Probab.}, 41(3A):1585--1627, 2013.
	
	\bibitem{CPU17}
	M.~E. Caballero, J.~L. P\'{e}rez~Garmendia, and G.~Uribe~Bravo.
	\newblock Affine processes on {$\mathbb R_+^m\times\mathbb R^n$} and
	multiparameter time changes.
	\newblock {\em Ann. Inst. Henri Poincar\'{e} Probab. Stat.}, 53(3):1280--1304,
	2017.
	
	\bibitem{Clancy19}
	D.~{Clancy, Jr}.
	\newblock {The Gorin-Shkolnikov identity and its random tree generalization}.
	\newblock {\em arXiv e-prints}, page arXiv:1910.08672, Oct 2019.
	
	\bibitem{CKG20}
	G.~{Conchon--Kerjan} and C.~{Goldschmidt}.
	\newblock {The stable graph: the metric space scaling limit of a critical
		random graph with i.i.d. power-law degrees}.
	\newblock {\em arXiv e-prints}, page arXiv:2002.04954, Feb. 2020.
	
	\bibitem{DKLP10}
	J.~Ding, J.~H. Kim, E.~Lubetzky, and Y.~Peres.
	\newblock Diameters in supercritical random graphs via first passage
	percolation.
	\newblock {\em Combin. Probab. Comput.}, 19(5-6):729--751, 2010.
	
	\bibitem{DKLP11}
	J.~Ding, J.~H. Kim, E.~Lubetzky, and Y.~Peres.
	\newblock Anatomy of a young giant component in the random graph.
	\newblock {\em Random Structures Algorithms}, 39(2):139--178, 2011.
	
	\bibitem{DL06}
	R.~G. Dolgoarshinnykh and S.~P. Lalley.
	\newblock Critical scaling for the {SIS} stochastic epidemic.
	\newblock {\em J. Appl. Probab.}, 43(3):892--898, 2006.
	
	\bibitem{Duquesne09}
	T.~Duquesne.
	\newblock Continuum random trees and branching processes with immigration.
	\newblock {\em Stochastic Process. Appl.}, 119(1):99--129, 2009.
	
	\bibitem{LD02}
	T.~Duquesne and J.-F. Le~Gall.
	\newblock Random trees, {L}\'{e}vy processes and spatial branching processes.
	\newblock {\em Ast\'{e}risque}, (281):vi+147, 2002.
	
	\bibitem{ER60}
	P.~Erd\H{o}s and A.~R\'{e}nyi.
	\newblock On the evolution of random graphs.
	\newblock {\em Magyar Tud. Akad. Mat. Kutat\'{o} Int. K\"{o}zl.}, 5:17--61,
	1960.
	
	\bibitem{EK86}
	S.~N. Ethier and T.~G. Kurtz.
	\newblock {\em Markov processes}.
	\newblock Wiley Series in Probability and Mathematical Statistics: Probability
	and Mathematical Statistics. John Wiley \& Sons, Inc., New York, 1986.
	\newblock Characterization and convergence.
	
	\bibitem{GHS18}
	C.~{Goldschmidt}, B.~{Haas}, and D.~{S{\'e}nizergues}.
	\newblock {Stable graphs: distributions and line-breaking construction}.
	\newblock {\em arXiv e-prints}, page arXiv:1811.06940, Nov 2018.
	
	\bibitem{GS18}
	V.~Gorin and M.~Shkolnikov.
	\newblock Stochastic {A}iry semigroup through tridiagonal matrices.
	\newblock {\em Ann. Probab.}, 46(4):2287--2344, 2018.
	
	\bibitem{JS87}
	J.~Jacod and A.~N. Shiryaev.
	\newblock {\em Limit theorems for stochastic processes}, volume 288 of {\em
		Grundlehren der Mathematischen Wissenschaften [Fundamental Principles of
		Mathematical Sciences]}.
	\newblock Springer-Verlag, Berlin, 1987.
	
	\bibitem{LS19}
	P.~Y.~G. Lamarre and M.~Shkolnikov.
	\newblock Edge of spiked beta ensembles, stochastic {A}iry semigroups and
	reflected {B}rownian motions.
	\newblock {\em Ann. Inst. Henri Poincar\'{e} Probab. Stat.}, 55(3):1402--1438,
	2019.
	
	\bibitem{Lamperti67}
	J.~Lamperti.
	\newblock Continuous state branching processes.
	\newblock {\em Bull. Amer. Math. Soc.}, 73:382--386, 1967.
	
	\bibitem{LL98b}
	J.-F. Le~Gall and Y.~Le~Jan.
	\newblock Branching processes in {L}\'{e}vy processes: {L}aplace functionals of
	snakes and superprocesses.
	\newblock {\em Ann. Probab.}, 26(4):1407--1432, 1998.
	
	\bibitem{LL98a}
	J.-F. Le~Gall and Y.~Le~Jan.
	\newblock Branching processes in {L}\'{e}vy processes: the exploration process.
	\newblock {\em Ann. Probab.}, 26(1):213--252, 1998.
	
	\bibitem{Luczak98}
	T.~{\L}uczak.
	\newblock Random trees and random graphs.
	\newblock In {\em Proceedings of the {E}ighth {I}nternational {C}onference
		``{R}andom {S}tructures and {A}lgorithms'' ({P}oznan, 1997)}, volume~13,
	pages 485--500, 1998.
	
	\bibitem{LPW94}
	T.~{\L}uczak, B.~Pittel, and J.~C. Wierman.
	\newblock The structure of a random graph at the point of the phase transition.
	\newblock {\em Trans. Amer. Math. Soc.}, 341(2):721--748, 1994.
	
	\bibitem{MartinLof98}
	A.~Martin-L\"{o}f.
	\newblock The final size of a nearly critical epidemic, and the first passage
	time of a {W}iener process to a parabolic barrier.
	\newblock {\em J. Appl. Probab.}, 35(3):671--682, 1998.
	
	\bibitem{RY99}
	D.~Revuz and M.~Yor.
	\newblock {\em Continuous martingales and Brownian motion}, volume 293 of {\em
		Grundlehren der Mathematischen Wissenschaften [Fundamental Principles of
		Mathematical Sciences]}.
	\newblock Springer-Verlag, Berlin, third edition, 1999.
	
	\bibitem{RW10}
	O.~Riordan and N.~Wormald.
	\newblock The diameter of sparse random graphs.
	\newblock {\em Combin. Probab. Comput.}, 19(5-6):835--926, 2010.
	
	\bibitem{Silverstein67}
	M.~L. Silverstein.
	\newblock A new approach to local times.
	\newblock {\em J. Math. Mech.}, 17:1023--1054, 1967/1968.
	
	\bibitem{Simatos15}
	F.~Simatos.
	\newblock State space collapse for critical multistage epidemics.
	\newblock {\em Adv. in Appl. Probab.}, 47(3):715--740, 2015.
	
	\bibitem{vBML80}
	B.~von Bahr and A.~Martin-L\"{o}f.
	\newblock Threshold limit theorems for some epidemic processes.
	\newblock {\em Adv. in Appl. Probab.}, 12(2):319--349, 1980.
	
	\bibitem{Whitt80}
	W.~Whitt.
	\newblock Some useful functions for functional limit theorems.
	\newblock {\em Math. Oper. Res.}, 5(1):67--85, 1980.
	
	\bibitem{YW71}
	T.~Yamada and S.~Watanabe.
	\newblock On the uniqueness of solutions of stochastic differential equations.
	\newblock {\em J. Math. Kyoto Univ.}, 11:155--167, 1971.
	
\end{thebibliography}

\end{document}